\documentclass{amsart}

\usepackage[utf8]{inputenc}
\usepackage[T1]{fontenc}
\usepackage[english,croatian]{babel}

\usepackage{lmodern}

\usepackage{amsmath}
\usepackage{amsfonts}
\usepackage{amssymb}
\usepackage{amsthm}

\usepackage{color}
\usepackage{float}
\usepackage{wrapfig}
\usepackage{graphicx}
\usepackage{url}

\usepackage{pdflscape}
\usepackage{verbatim}
\usepackage{etoolbox}

\newcommand{\mbbZ}{\mathbb{Z}}

\newcommand{\mbbC}{\mathbb{C}}

\newcommand{\mbbP}{\mathbb{P}}

\renewcommand{\dim}{\operatorname{dim}}
\newcommand{\lspan}{\operatorname{span}}
\newcommand{\Ker}{\operatorname{Ker}}
\newcommand{\Coker}{\operatorname{Coker}}
\renewcommand{\Im}{\operatorname{Im}}
\newcommand{\Hom}{\operatorname{Hom}}

\newcommand{\Diff}{\operatorname{Diff}}
\renewcommand{\exp}{\operatorname{exp}}

\newcommand{\End}{\operatorname{End}}

\newcommand{\GL}{\operatorname{GL}}

\newcommand{\Sp}{\operatorname{Sp}}

\newcommand{\iGr}{\operatorname{iGr}}

\newcommand{\sign}{\operatorname{sign}}

\newcommand{\mcaO}{\mathcal{O}}

\newcommand{\mcaF}{\mathcal{F}}

\newcommand{\mfrb}{\mathfrak{b}}
\newcommand{\mfrg}{\mathfrak{g}}
\newcommand{\mfrh}{\mathfrak{h}}

\newcommand{\mfrl}{\mathfrak{l}}

\newcommand{\mfrp}{\mathfrak{p}}

\newcommand{\mfrr}{\mathfrak{r}}
\newcommand{\mfru}{\mathfrak{u}}
\newcommand{\mfrgl}{\mathfrak{gl}}

\newcommand{\mfrsp}{\mathfrak{sp}}

\newcommand{\pbar}{\, |\,}

\newtheorem{theorem}{Theorem}
\newtheorem*{theorem*}{Theorem}
\newtheorem{proposition}[theorem]{Proposition}
\newtheorem*{proposition*}{Proposition}
\newtheorem{lemma}[theorem]{Lemma}
\newtheorem*{lemma*}{Lemma}
\newtheorem{corollary}[theorem]{Corollary}
\newtheorem*{corollary*}{Corollary}
\theoremstyle{remark}
\theoremstyle{remark}\newtheorem*{definition*}{Definition}
\theoremstyle{remark}
\theoremstyle{remark}\newtheorem*{example*}{Example}
\theoremstyle{remark}\newtheorem{remark}[theorem]{Remark}
\theoremstyle{remark}\newtheorem*{remark*}{Remark}
\theoremstyle{remark}
\theoremstyle{remark}\newtheorem*{conjecture*}{Conjecture}

\usepackage[all]{xy}

\usepackage{tikz}
\usetikzlibrary{fit,matrix,positioning,calc}
\usepackage{tkz-graph}

\newlength{\defaultboxsize}
\usepackage[boxsize=\defaultboxsize,aligntableaux=center]{ytableau}

\RequirePackage{tikz}
\usetikzlibrary{decorations.markings}

\newenvironment{dynkin}
{\begin{tikzpicture}[baseline={(0,{-0.7*height("$\alpha_{1_1}$")*1pt})}, decoration={markings,mark=at position 0.6 with {\arrow[ultra thick]{<}}}]}
{\end{tikzpicture}}

\newlength{\dynkinstep}
\newlength{\dynkindotradius}
\newlength{\dynkincrosssize}
\newcommand{\dynkinline}[4]{
\draw(\dynkinstep*#1,\dynkinstep*#2) -- (\dynkinstep*#3,\dynkinstep*#4);}

\newcommand{\dynkindoubleline}[4]{
\draw[double,double distance=.7*\dynkindotradius,postaction={decorate}] (\dynkinstep*#1,\dynkinstep*#2) -- (\dynkinstep*#3,\dynkinstep*#4);}

\newcommand{\dynkindot}[2]{
\draw (\dynkinstep*#1,\dynkinstep*#2) [fill=white] circle  (\dynkindotradius);}

\newcommand{\dynkincross}[2]{
\draw[very thick,white] (#1*\dynkinstep-\dynkincrosssize,#2*\dynkinstep-\dynkincrosssize) -- (#1*\dynkinstep+\dynkincrosssize,#2*\dynkinstep+\dynkincrosssize);
\draw[very thick,white] (#1*\dynkinstep-\dynkincrosssize,#2*\dynkinstep+\dynkincrosssize) -- (#1*\dynkinstep+\dynkincrosssize,#2*\dynkinstep-\dynkincrosssize);
\draw (#1*\dynkinstep-\dynkincrosssize,#2*\dynkinstep-\dynkincrosssize) -- (#1*\dynkinstep+\dynkincrosssize,#2*\dynkinstep+\dynkincrosssize);
\draw (#1*\dynkinstep-\dynkincrosssize,#2*\dynkinstep+\dynkincrosssize) -- (#1*\dynkinstep+\dynkincrosssize,#2*\dynkinstep-\dynkincrosssize);}

\newcommand{\dynkindots}[4]{
\draw[dotted] (\dynkinstep*#1,\dynkinstep*#2) -- (\dynkinstep*#3,\dynkinstep*#4);}

\newcommand{\dynkinlabel}[4]{
\node [#3] at (\dynkinstep*#1,\dynkinstep*#2) {{\scriptsize #4}};}

\newcommand\halfbox[1]{ \tikz[baseline=(n.base)]{
\node(n)[inner sep=1pt]{$#1$};
\draw[line cap=round](n.north west)--(n.south west)--(n.south east);}\smallskip}

\newcommand\spectralsequence[1]{\setlength{\abovedisplayskip}{1em}
\setlength{\belowdisplayskip}{1em}
\halfbox{\xymatrix@R=.3em@C=.9em{#1}}}

\makeatletter
\renewcommand*\env@matrix[1][*\c@MaxMatrixCols c]{%
  \hskip -\arraycolsep
  \let\@ifnextchar\new@ifnextchar
  \array{#1}}
\makeatother

\usepackage{enumitem}
\setlist[enumerate,1]{label=(\alph*), ref=(\alph*)}
\setlist[enumerate,2]{label=\roman*., ref=\roman*.}
\setlist[enumerate,3]{label=\Alph*., ref=\Alph*.}
\setlist[enumerate,4]{label=\arabic*., ref=\arabic*.}

\allowdisplaybreaks[1] 

\usepackage{caption}

\usepackage{mathtools}
\usepackage{hyperref}

\begin{document}

\selectlanguage{english}

\title{Singular BGG complexes over isotropic 2-Grassmannian}

\author{Denis Husadžić}

\address{Faculty of Science, University of Zagreb, Bijenička cesta 30, 10 000 Zagreb, Croatia.}

\email{dhusadzi@math.hr}

\author{Rafael Mrđen}

\address{Faculty of Civil Engineering, University of Zagreb, Fra Andrije Kačića-Miošića 26, 10 000 Zagreb, Croatia.}

\email{rafaelm@grad.hr}

\thanks{The authors acknowledge support from the Croatian Science Foundation grant no. 4176, and the QuantiXLie Center of Excellence grant no. KK.01.1.1.01.0004 funded by the European Regional Development Fund.}




\begin{abstract}
We construct exact sequences of invariant differential operators acting on sections of certain homogeneous vector bundles in singular infinitesimal character, over the isotropic $2$-Grassmannian. This space is equal to $G/P$, where $G$ is $\Sp(2n,\mbbC)$, and $P$ its standard parabolic subgroup having the Levi factor $\GL(2,\mbbC) \times \Sp(2n-4,\mbbC)$. The constructed sequences are analogues of the Bernstein-Gelfand-Gelfand resolutions. We do this by considering the Penrose transform over an appropriate double fibration. The result differs from the Hermitian situation.
\end{abstract}

\maketitle


\section{Introduction and preliminaries}

The BGG (Bernstein-Gelfand-Gelfand) complexes first appeared as certain resolutions of irreducible finite dimensional $\mfrg$-modules by a direct sums of generalized Verma modules of fixed type $(\mfrg,\mfrp)$, where $\mfrp \subseteq \mfrg$ is a parabolic subalgebra of a complex semisimple Lie algebra (\cite{bernstein1975differential}, \cite{lepowsky1977generalization}). The highest weights of the generalized Verma modules appearing in the resolution are exactly the Levi-dominant elements in the affine Weyl group orbit of the highest weight of the resolved module. 

In the dual geometric picture (Remark \ref{remark:duality}), BGG complexes were studied by Čap, Slovák and Souček (\cite{cap2001bernstein}). They constructed BGG complexes in a very general theory of ``curved'' parabolic geometries. In the flat model for this theory (homogeneous space $G/P$, where $P \subseteq G$ is a parabolic subgroup of a complex semisimple Lie group), their construction yields a locally exact resolution of the constant sheaf over $G/P$ defined by a finite dimensional $G$-module, by direct sums of homogeneous vector bundles and invariant differential operators; essentially dual to the one in \cite{lepowsky1977generalization}. The BGG complexes show up in many different areas of mathematics (\cite{eastwood1999variations}). For example, the BGG complex for the trivial $G$-module appears as a subcomplex of the holomorphic de Rham complex, and coincides with it precisely when $G/P$ is a Hermitian symmetric space.

One important requirement in the constructions mentioned above is that the modules are of regular infinitesimal character, so the resolved module is a finite-dimensional one. There are no general constructions of analogous resolutions in a singular infinitesimal character. A serious obstacle there is a lack of the so called standard operators. The non-standard operators have not been classified yet, not even in  regular infinitesimal character, except in some special cases (e.g. \cite{matumoto2006homomorphisms}).

The question of existence of such resolutions in singular infinitesimal character, at least in principle, is settled for all Hermitian pairs (also known as $|1|$-graded, i.e. parabolics whose nilpotent radical is abelian), by the Enright-Shelton theory (\cite{enright1987categories}). However, from this theory it is not clear how to construct the operators appearing in the resolutions. It turned out that the Penrose transform (as in \cite{baston2016penrose}) is a particularly useful tool for the construction of such operators, as observed by Baston in \cite{baston1992quaternionic}, by working on (singular) quaternionic complexes. In \cite{pandzic2016bgg}, Pandžić and Souček constructed singular BGG complexes in type $A$, for all maximal parabolics, i.e., all complex Grassmannians. In \cite{mrden2017singular}, the author obtained analogous resolutions over the Langrangian Grassmannian, in a semi-regular infinitesimal character. All these cases above are Hermitian. It is visible from the construction that either the whole singular orbit makes a BGG complex, or (if the long simple root is singular) the orbit decomposes into two subsets, giving two disjoint BGG complexes, as predicted by the Enright-Shelton theory.

It would be of interest to provide singular BGG complexes for all maximal parabolics in type $C$. However, all but one are $|2|$-graded, hence the situation is much more complicated than in the Hermitian cases. In particular, the Enright-Shelton theory is not applicable.

This paper deals with the isotropic $2$-Grassmannian, with respect to a skew-symmetric form. This is a (quotient by a) maximal parabolic subgroup in type $C$, denoted by $\begin{dynkin}
\dynkinline{1}{0}{2}{0}
\dynkinline{2}{0}{3}{0}
\dynkindots{3}{0}{4}{0}
\dynkindoubleline{4}{0}{5}{0}
\dynkindot{1}{0}
\dynkincross{2}{0}
\dynkindot{3}{0}
\dynkindot{4}{0}
\dynkindot{5}{0}
\end{dynkin}$. Also, it is the flat model for quaternionic contact geometries (\cite[4.3.3.]{cap2009parabolic}). The main result, Theorem \ref{theorem:main}, gives the BGG complexes in many singular infinitesimal characters, which cover almost all types of singularities for this space. The exactness over the big affine cell is proved. The construction of the non-standard differential operators appearing in the resolutions is also described. It turns out that our singular orbits consist of two BGG complexes that have an object in common, in contrast to the Hermitian situations.

The question of existence of BGG complexes is closely related to classification of Kostant modules (\cite{boe2009kostant}, \cite{enright2014diagrams}), which is solved only in the Hermitian cases. Also, BGG complexes are closely related to the Cousin complexes, as in \cite{milicic_cousin}. For other similar results in a higher grading, see \cite{krump2006singular}, \cite{salac2018k-dirac}, \cite{salac2018resolution}.

\subsection{Preliminaries}

Let $G$ be a semisimple complex Lie group, connected and simply connected, $\mfrg$ its Lie algebra, $\mfrh$ its fixed Cartan subalgebra, and $\Delta^+(\mfrg,\mfrh)$ a fixed set of positive roots. The half sum of all the positive roots will be denoted by $\rho$. For an element $w \in W_\mfrg$ of the Weyl group, denote by $l(w)$ the length of $w$, i.e., the minimal number of simple reflections required to obtain $w$. For $w, w' \in W_\mfrg$ we write $w \stackrel{\alpha}{\longrightarrow} w'$ if $l(w')=l(w)+1$ and $w'=\sigma_\alpha \circ w$ (where $\sigma_\alpha$ is the reflection with respect to $\alpha$), for some $\alpha \in \Delta^+(\mfrg,\mfrh)$, not necessarily simple. We often write only $w \longrightarrow w'$. In this way, $W_\mfrg$ becomes a directed graph. Besides the standard action of $W_\mfrg$ on $\mfrh^\ast$, we also need the affine action: $w \cdot \lambda = w(\lambda+\rho)-\rho$. In fact, we will almost always write weights in the $\rho$-shifted coordinates, and apply the standard action of $W_\mfrg$. 

Fix a standard parabolic subalgebra $\mfrp = \mfrl \oplus \mfru$ (the Levi decomposition) of $\mfrg$. The Hasse diagram of $\mfrp$, denoted by $W^\mfrp$, is the full subgraph of $W_\mfrg$ consisting of all elements in $W_\mfrg$ that map $\mfrg$-dominant weights to $\mfrl$-dominant ones (or equivalently, that map $\rho$ to a $\mfrl$-dominant weight). For a $\mfrg$-integral and $\mfrl$-dominant weight $\lambda$, we write $F_\mfrp(\lambda)$ for the finite-dimensional, irreducible representation of $\mfrl$ with highest weight $\lambda$, and with $\mfru$ acting by $0$. We write $E_\mfrp(\lambda)$ for its dual. The same notation is used for the group representations.

Given a finite-dimensional holomorphic representation $\pi \colon P \to \End(V)$, we can form the homogeneous holomorphic vector bundle $G \times_P V \to G/P$. Its holomorphic sections correspond to $V$-valued holomorphic functions on open subsets of $G$ that are $P$-equivariant, and they make up a homogeneous sheaf. For $V=E_\mfrp(\lambda)$, this sheaf is denoted by $\mcaO_\mfrp(\lambda)$. By an invariant differential operator we will mean a $\mbbC$-linear differential operator $\mcaO_\mfrp(\lambda) \to \mcaO_\mfrp(\mu)$, invariant with respect to the left translation of sections.

Consider the Borel subgroup $B \subseteq P$. If there exists a non-zero invariant differential operator $\mcaO_\mfrb(\lambda) \to \mcaO_\mfrb(\mu)$, then it is unique up to a scalar. The direct image of such a map via $G/B \to G/P$ is again an invariant differential operator, called the standard operator $\mcaO_\mfrp(\lambda) \to \mcaO_\mfrp(\mu)$. It may be zero, and there may exist invariant differential operators which are non-standard. Standard operators are in principle completely known, but non-standard ones have not yet been classified. In this paper we will construct some non-standard operators. 

\begin{remark}
\label{remark:duality}
There is a well known contravariant correspondence between the sheaves $\mcaO_\mfrp(\lambda)$ and the generalized Verma modules $M_\mfrp(\lambda) = U(\mfrg) \otimes_{U(\mfrp)} F_\mfrp(\lambda)$:
\begin{equation*}
\label{equation:inv_diff_op=verma_hom}
\Diff_G(\mcaO_\mfrp(\lambda),\mcaO_\mfrp(\mu)) \cong \Hom_\mfrg(M_\mfrp(\mu),M_\mfrp(\lambda)),
\end{equation*}
where the left-hand side denotes invariant differential operators. See \cite[A]{cap2001bernstein}. The correspondence enables us to study invariant differential operators via algebraic and representation-theoretic machinery.
\end{remark}

\begin{theorem*}[Bernstein-Gelfand-Gelfand-Lepowsky, Čap-Slov\'{a}k-Souček]
\label{theorem:regular_BGG_geometric}
For any $\mfrg$-integral and $\mfrg$-dominant weight $\lambda$, there is a locally exact sequence on $G/P$ resolving the constant sheaf defined by $E_\mfrg(\lambda)$, called the (regular) BGG resolution:
\begin{equation*}
\label{equation:regular_BGG}
0 \to E_\mfrg(\lambda) \to \Delta^\bullet(\lambda), \quad \text{where} \quad \Delta^k(\lambda)= \bigoplus_{w \in W^{\mfrp}, \ l(w)=k} \mcaO_\mfrp(w \cdot \lambda).
\end{equation*}
The morphisms are the direct sums of the standard operators $\mcaO_\mfrp(w \cdot \lambda) \to \mcaO_\mfrp(w' \cdot \lambda)$ for $w \rightarrow w'$ in $W^\mfrp$, all of which are non-zero. 
\end{theorem*}

The modules in the theorem above are all of infinitesimal character $\lambda+\rho$, hence regular. We are interested in finding the analogues of the above theorem in singular infinitesimal characters.

\subsection{Type C}
We specialize to $G=\Sp(2n,\mbbC)=\begin{dynkin}
\setlength{\dynkinstep}{1cm}
\dynkinline{1}{0}{2}{0}
\dynkindots{2}{0}{3}{0}
\dynkinline{3}{0}{4}{0}
\dynkindoubleline{4}{0}{5}{0}
\dynkindot{1}{0}
\dynkinlabel{1}{0}{below}{$\alpha_1$}
\dynkindot{2}{0}
\dynkinlabel{2}{0}{below}{$\alpha_2$}
\dynkindot{3}{0}
\dynkinlabel{3}{0}{below}{$\alpha_{n-2}$}
\dynkindot{4}{0}
\dynkinlabel{4}{0}{below}{$\alpha_{n-1}$}
\dynkindot{5}{0}
\dynkinlabel{5}{0}{below}{$\alpha_n$}
\end{dynkin}$, the complex symplectic group: linear operators on $\mbbC^{2n}$ preserving the bilinear skew-symmetric form $J=\begin{pmatrix}
0 & I_n \\ -I_n & 0
\end{pmatrix}$. Choose the Cartan subalgebra of the Lie algebra of $G$ consisting of diagonal matrices $\mfrh \subseteq \mfrg=\mfrsp(2n,\mbbC)$, and the positive roots:
\[ \Delta^+(\mfrg,\mfrh) = \{  a_{ij} := \epsilon_i - \epsilon_j , \quad b_i := 2 \epsilon_i , \quad c_{ij} :=  \epsilon_i+ \epsilon_j \ \colon \ 1\leq i < j \leq n \},
\]
where $\epsilon_i$ denotes the projection to the $i$-th coordinate. The set $\Pi$ of the simple roots consists of short roots $\alpha_i=a_{i,i+1}$ for $i=1,\ldots,n-1$, and the long root $\alpha_n=b_n$. A weight $\lambda = (\lambda_1, \lambda_2, \ldots, \lambda_n ) \in \mfrh^\ast$ is integral if all $\lambda_i \in \mbbZ$, and dominant if $\lambda_1 \geq \lambda_2 \geq \ldots \lambda_n \geq 0$. In this paper we only work with integral weights. The half sum of all positive roots is $\rho = (n,n-1,\ldots,1)$. The simple reflection $\sigma_{a_{ij}}$ acts on $\mfrh^\ast$ by transposition of the $i$-th and the $j$-th coordinate, $\sigma_{b_i}$ changes sign of the $i$-th coordinate, and $\sigma_{c_{ij}}$ acts as transposition and sign changes on the $i$-th and the $j$-th coordinates. So, the Weyl group acts by permutations and sign changes of the coordinates.  A weight is regular if and only if it does not have two coordinates with the same absolute value, and all the coordinates are non-zero; otherwise it is singular. A singular weight is called semi-regular, if it is orthogonal to only one simple root.

Parabolic subgroups/subalgebras will be given by crossing some nodes on the Dynkin diagram for $\mfrg$. The root vectors corresponding to the uncrossed nodes and their negatives, and the Cartan subalgebra, generate the Levi factor $\mfrl$. The root vectors of the positive roots not contained in $\Delta(\mfrl,\mfrh)$ generate the nilpotent radical $\mfru$. In this way we get a standard parabolic subalgebra $\mfrp = \mfrl \oplus \mfru$. The corresponding parabolic subgroup $P$ is defined as the normalizer of $\mfrp$ in $G$. There is also the opposite nilpotent radical $\mfru^-$, consisting of root subspaces opposite to those in $\mfru$. Then, we have $\mfrg= \mfru^- \oplus \mfrl \oplus \mfru$. Weights for the Levi factor can be written as $n$-tuples again, but for every crossed node $\alpha_i$ in the Dynkin diagram for the parabolic subalgebra, we will put a bar after the $i$-th coordinate of the weight. This is a good mnemonic for keeping track of the dominance condition for the Levi factor: there is no condition on the order of an adjacent pair of coordinates if there is a bar between these coordinates. But the coordinates have to descend in each group, and also be strictly positive in the last group. If there is a bar after the last coordinate, then there is no positivity condition.

Suppose $P \subseteq G=\Sp(2n,\mbbC)$ is a standard parabolic subgroup, say it has crosses on the positions $k_1, \ldots, k_s$ in its Dynkin diagram. A geometric realization of the generalized flag manifold $G/P$ is given by
\begin{equation}
\label{equation:flags}
\begin{split}
G/P \cong \big\{ ( W_1, W_2, \ldots, W_s) \ \colon \ &W_1 \leq \ldots \leq W_s \leq V, \\ &\dim W_i = k_i \text{ for } i=1,\ldots,s, \ W_s \text{ isotropic} \big\},
\end{split}
\end{equation}
the space of isotropic partial flags of type $(k_1,\ldots,k_s)$. Here $V$ is (the space of) the standard representation of $\Sp(2n,\mbbC)$.

In a special case when $P$ is a maximal parabolic subgroup, say it has a cross on the position $k$, $G/P$ is realized as the space of all $k$-dimensional isotropic subspaces in $V$, also known as the isotropic $k$-Grassmannian, and denoted by $\iGr(k,2n)$.

\subsection{Isotropic 1-Grassmannian}
Let us first consider $\mfrp =\begin{dynkin}
\dynkinline{1}{0}{2}{0}
\dynkindots{2}{0}{3}{0}
\dynkindoubleline{3}{0}{4}{0}
\dynkincross{1}{0}
\dynkindot{2}{0}
\dynkindot{3}{0}
\dynkindot{4}{0}
\end{dynkin}$. The regular Hasse diagram $W^\mfrp$ is easily found using \cite[Proposition 3.2.14.]{cap2009parabolic}. Here we give the orbit $W^\mfrp \rho$, including the ``arrow'' relation, and therefore also the ($\rho$-shifted) weights parametrizing the BGG resolution over $\iGr(1,2n)$ in the infinitesimal character $\rho$:
\begin{multline}
\label{equation:BGG_iGr_1}
(n \pbar n-1,\ldots,1) \rightarrow (n-1 \pbar n,n-2,\ldots,1) \rightarrow \ldots  \rightarrow (1 \pbar n,n-1,\ldots,2) \rightarrow \\
\rightarrow (-1 \pbar n,\ldots,2) \rightarrow (-2 \pbar n,\ldots,3,1) \rightarrow \ldots \rightarrow (-n \pbar n-1,\ldots,1) \to 0.
\end{multline}
All the operators in the resolution are of order one except $(1\pbar \ldots) \to (-1\pbar \ldots)$, which is of order two.

An infinitesimal character can be singular if its $\mfrg$-dominant representative has some repeated coordinates, or if it has a zero among its coordinates. Such a weight can have a strictly $\mfrl$-dominant $\mfrg$-conjugate only if it is semi-regular, that is, exactly one pair of coordinates coincide, or there is exactly one zero coordinate (but not both). Considering only minimal integral weights (only for notational purpose, the same calculation works for any semi-regular infinitesimal character), these are:
\begin{gather}
\label{definition:lambda_k}\lambda_k + \rho := (n-1,n-2,\ldots,k+1,k,k,k-1,\ldots,2,1), \quad 1 \leq k \leq n-1, \\
\label{definition:lambda_0}\lambda_0 + \rho := (n-1,n-2,\ldots,2,1,0),
\end{gather}
and their strictly $\mfrl$-dominant conjugates are the following:
\begin{equation}
\label{definition:lambda_k_conj}
\tilde{\lambda}_k^{\pm} + \rho := (\pm k \pbar n-1,n-2, \ldots, 1), \quad 0 \leq k \leq n-1.
\end{equation}
Since $\tilde{\lambda}^+_0=\tilde{\lambda}^-_0$, the corresponding generalized Verma module is simple (there are no other strictly $\mfrl$-dominant conjugates). But in principle, there could be a homomorphism $M_\mfrp(\tilde{\lambda}^-_k) \to M_\mfrp(\tilde{\lambda}^+_k)$. However, a simple calculation using Jantzen simplicity criterion shows that these modules are in fact simple, so such a homomorphism is trivial (see \cite{matumoto2006homomorphisms}). By Remark \ref{remark:duality}, as a consequence we have:
\begin{corollary}
There are no non-trivial singular BGG complexes over $\iGr(1,2n)$.
\end{corollary}

Note that $G/P \cong \iGr(1,2n)$ is diffeomorphic to the complex projective space $\mbbP^{2n-1}$, because our form is skew-symmetric. But it is not isomorphic to it as a homogeneous space. The space $\iGr(1,2n)$ is the flat model for projective contact geometries (\cite[4.2.6.]{cap2001bernstein}).

\section{Regular and singular orbits}

From now on, we fix the standard parabolic subalgebra
\[ \mfrp=\mfrl \oplus \mfru =\begin{dynkin}
\dynkinline{1}{0}{2}{0}
\dynkinline{2}{0}{3}{0}
\dynkindots{3}{0}{4}{0}
\dynkinline{4}{0}{5}{0}
\dynkindoubleline{5}{0}{6}{0}
\dynkindot{1}{0}
\dynkincross{2}{0}
\dynkindot{3}{0}
\dynkindot{4}{0}
\dynkindot{5}{0}
\dynkindot{6}{0}
\end{dynkin}, \]
which has the Levi factor $\mfrl = \mfrgl(2,\mbbC) \oplus \mfrsp(2n-4,\mbbC)$, and the following roots:
\begin{align}
\nonumber &\Delta^+(\mfrl,\mfrh) =  \{  a_{12} \} \cup \{ a_{ij}, \ b_i, \ c_{ij} \ \colon \ 3 \leq i<j \leq n \}, \\
\label{equation:Delta_u}
\Delta(\mfru) = \{ & a_{13}, \ldots, a_{1n}, \ a_{23}, \ldots, a_{2n}, \ c_{13}, \ldots, c_{1n}, \ c_{23}, \ldots, c_{2n} \\
\nonumber & \ b_1, \ b_2, \ c_{12} \}.
\end{align}
Note that the roots $b_1$, $b_2$ and $c_{12}$ in $\Delta(\mfru)$ have the property that the crossed simple root appears twice in their decompositions. Because of this, $\mfru$ is not abelian, and $G/P \cong \iGr(2,2n)$ does not have a structure of a Hermitian symmetric space. However, there is a $|2|$-grading of $\mfrg$, given by the eigenspaces of the grading element $E=(1,1,0, \ldots, 0) \in \mfrh$, so that $\mfrg_0=\mfrl$, $\mfrg_1 \oplus \mfrg_2 = \mfru$ and  $\mfrg_{-1} \oplus \mfrg_{-2} = \mfru^-$. For details, see \cite[3.2.1]{cap2009parabolic}.

We now describe the regular Hasse diagram $W^\mfrp$ for this parabolic subalgebra, by giving the $\mfrl$-dominant part of the orbit of $\rho$ under the Weyl group, respecting the ``arrow'' relation. First note that if $w \rho = (\mu_1,\mu_2 \pbar \mu_3, \ldots, \mu_n)$ is $\mfrl$-dominant, then $\mu_1,\mu_2 \in \{\pm 1, \pm 2, \ldots, \pm n\}$, $\mu_1 > \mu_2$, $\mu_1 \neq - \mu_2$, and the rest of the coordinates are completely determined, since they must be strictly decreasing and positive. By projecting the orbit onto the first two coordinates, we get the picture in Figure \ref{figure:regular_orbit}.
\begin{figure}
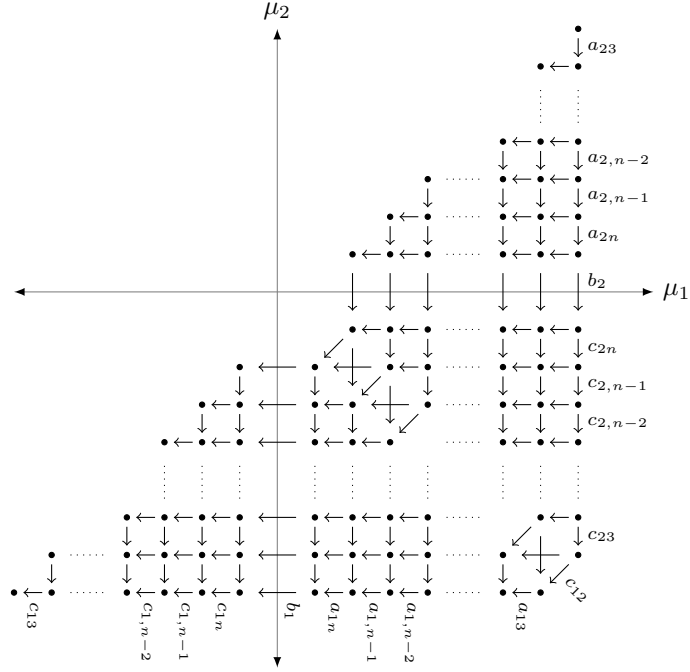

\begin{center}
\tikz[x=.5cm,y=.5cm]{ 

\def\n{8} 
\def\skipa{5} 
\def\skipb{11} 

\draw[latex-latex, thin, draw=gray] (-\n+1,0)--(\n+2,0) node [right] {$\mu_1$}; 
\draw[latex-latex, thin, draw=gray] (0,-\n-2)--(0,\n-1) node [above] {$\mu_2$}; 

\newcommand{\ldominant}[2]{\draw[fill] (#1,#2) circle [radius=1pt];} 

\newcommand{\arrow}[4]{\draw [->] (.75*#1+.25*#3,.75*#2+.25*#4) -- (.25*#1+.75*#3,.25*#2+.75*#4);} 

\newcommand{\skipped}[4]{\draw [dotted,-] (.75*#1+.25*#3,.75*#2+.25*#4) -- (.25*#1+.75*#3,.25*#2+.75*#4);} 

\foreach \x in {1,...,\n} \foreach \y in {1,...,\n} {
	\ifthenelse{\not \x=\skipa \and \not \x=\skipb \and \not \y=\skipa \and \not \y=\skipb}{
		\ifthenelse{\x>\y}{\ldominant{\x}{\y};}{}; 
		\ifthenelse{\x=\y}{}{\ldominant{\x}{-\y};}; 
		\ifthenelse{\x<\y}{\ldominant{-\x}{-\y};}{}; 
	}{}
}

\arrow{3}{2}{3}{1};\arrow{4}{2}{4}{1};\arrow{4}{3}{4}{2};\arrow{6}{2}{6}{1};\arrow{6}{3}{6}{2};\arrow{6}{4}{6}{3};\arrow{7}{2}{7}{1};\arrow{7}{3}{7}{2};\arrow{7}{4}{7}{3};\arrow{8}{2}{8}{1};\arrow{8}{3}{8}{2};\arrow{8}{4}{8}{3};\arrow{8}{7}{8}{6};\arrow{3}{-1}{3}{-2};\arrow{4}{-1}{4}{-2};\arrow{4}{-2}{4}{-3};\arrow{6}{-1}{6}{-2};\arrow{6}{-2}{6}{-3};\arrow{6}{-3}{6}{-4};\arrow{7}{-1}{7}{-2};\arrow{7}{-2}{7}{-3};\arrow{7}{-3}{7}{-4};\arrow{8}{-1}{8}{-2};\arrow{8}{-2}{8}{-3};\arrow{8}{-3}{8}{-4};\arrow{8}{-6}{8}{-7};\arrow{1}{-2}{1}{-3};\arrow{1}{-3}{1}{-4};\arrow{1}{-6}{1}{-7};\arrow{1}{-7}{1}{-8};\arrow{2}{-3}{2}{-4};\arrow{2}{-6}{2}{-7};\arrow{2}{-7}{2}{-8};\arrow{3}{-6}{3}{-7};\arrow{3}{-7}{3}{-8};\arrow{4}{-6}{4}{-7};\arrow{4}{-7}{4}{-8};\arrow{6}{-7}{6}{-8};\arrow{-1}{-2}{-1}{-3};\arrow{-1}{-3}{-1}{-4};\arrow{-1}{-6}{-1}{-7};\arrow{-1}{-7}{-1}{-8};\arrow{-2}{-3}{-2}{-4};\arrow{-2}{-6}{-2}{-7};\arrow{-2}{-7}{-2}{-8};\arrow{-3}{-6}{-3}{-7};\arrow{-3}{-7}{-3}{-8};\arrow{-4}{-6}{-4}{-7};\arrow{-4}{-7}{-4}{-8};\arrow{-6}{-7}{-6}{-8};\arrow{2}{1}{2}{-1};\arrow{3}{1}{3}{-1};\arrow{4}{1}{4}{-1};\arrow{6}{1}{6}{-1};\arrow{7}{1}{7}{-1};\arrow{8}{1}{8}{-1};\arrow{2}{-1}{2}{-3};\arrow{3}{-2}{3}{-4};\arrow{7}{-6}{7}{-8};\arrow{3}{1}{2}{1};\arrow{4}{1}{3}{1};\arrow{4}{2}{3}{2};\arrow{7}{1}{6}{1};\arrow{7}{2}{6}{2};\arrow{7}{3}{6}{3};\arrow{7}{4}{6}{4};\arrow{8}{1}{7}{1};\arrow{8}{2}{7}{2};\arrow{8}{3}{7}{3};\arrow{8}{4}{7}{4};\arrow{8}{6}{7}{6};\arrow{3}{-1}{2}{-1};\arrow{4}{-1}{3}{-1};\arrow{4}{-2}{3}{-2};\arrow{7}{-1}{6}{-1};\arrow{7}{-2}{6}{-2};\arrow{7}{-3}{6}{-3};\arrow{7}{-4}{6}{-4};\arrow{8}{-1}{7}{-1};\arrow{8}{-2}{7}{-2};\arrow{8}{-3}{7}{-3};\arrow{8}{-4}{7}{-4};\arrow{8}{-6}{7}{-6};\arrow{2}{-3}{1}{-3};\arrow{2}{-4}{1}{-4};\arrow{3}{-4}{2}{-4};\arrow{2}{-6}{1}{-6};\arrow{3}{-6}{2}{-6};\arrow{4}{-6}{3}{-6};\arrow{2}{-7}{1}{-7};\arrow{3}{-7}{2}{-7};\arrow{4}{-7}{3}{-7};\arrow{2}{-8}{1}{-8};\arrow{3}{-8}{2}{-8};\arrow{4}{-8}{3}{-8};\arrow{7}{-8}{6}{-8};\arrow{-1}{-3}{-2}{-3};\arrow{-1}{-4}{-2}{-4};\arrow{-2}{-4}{-3}{-4};\arrow{-1}{-6}{-2}{-6};\arrow{-2}{-6}{-3}{-6};\arrow{-3}{-6}{-4}{-6};\arrow{-1}{-7}{-2}{-7};\arrow{-2}{-7}{-3}{-7};\arrow{-3}{-7}{-4}{-7};\arrow{-1}{-8}{-2}{-8};\arrow{-2}{-8}{-3}{-8};\arrow{-3}{-8}{-4}{-8};\arrow{-6}{-8}{-7}{-8};\arrow{1}{-2}{-1}{-2};\arrow{1}{-3}{-1}{-3};\arrow{1}{-4}{-1}{-4};\arrow{1}{-6}{-1}{-6};\arrow{1}{-7}{-1}{-7};\arrow{1}{-8}{-1}{-8};\arrow{3}{-2}{1}{-2};\arrow{4}{-3}{2}{-3};\arrow{8}{-7}{6}{-7};\arrow{2}{-1}{1}{-2};\arrow{3}{-2}{2}{-3};\arrow{4}{-3}{3}{-4};\arrow{7}{-6}{6}{-7};\arrow{8}{-7}{7}{-8};
\skipped{4}{-8}{6}{-8};\skipped{4}{-7}{6}{-7};\skipped{4}{-6}{6}{-6};\skipped{4}{-4}{6}{-4};\skipped{4}{-3}{6}{-3};\skipped{4}{-2}{6}{-2};\skipped{4}{-1}{6}{-1};\skipped{4}{1}{6}{1};\skipped{4}{2}{6}{2};\skipped{4}{3}{6}{3};\skipped{-6}{-8}{-4}{-8};\skipped{-6}{-7}{-4}{-7};\skipped{7}{4}{7}{6};\skipped{8}{4}{8}{6};\skipped{-3}{-6}{-3}{-4};\skipped{-2}{-6}{-2}{-4};\skipped{-1}{-6}{-1}{-4};\skipped{1}{-6}{1}{-4};\skipped{2}{-6}{2}{-4};\skipped{3}{-6}{3}{-4};\skipped{4}{-6}{4}{-4};\skipped{6}{-6}{6}{-4};\skipped{7}{-6}{7}{-4};\skipped{8}{-6}{8}{-4};

\pgfmathtruncatemacro{\skipaa}{\skipa-1};
\node at (\n,\n-1-.5) [right] {\scriptsize $a_{23}$};
\foreach \x in {3,...,\skipaa} {	
	\pgfmathtruncatemacro{\xx}{\x-2};
	\node at (\n,\x-.5) [right] {\scriptsize $a_{2,n-\xx}$};
	}
\node at (\n,2-.5) [right] {\scriptsize $a_{2n}$};
\node at (\n,1-.7) [right] {\scriptsize $b_{2}$};
\node at (\n,-1-.5) [right] {\scriptsize $c_{2n}$};
\foreach \x in {3,...,\skipaa} {	
	\pgfmathtruncatemacro{\xx}{\x-2};
	\node at (\n,-\x+.5) [right] {\scriptsize $c_{2,n-\xx}$};
	}
\node at (\n,-\n+1+.5) [right] {\scriptsize $c_{23}$};
\node[rotate=-45] at (\n-1+.5,-\n+1-.5) [right] {\scriptsize $c_{12}$};
\node[rotate=-90] at (\n-1-.5,-\n) [right] {\scriptsize $a_{13}$};
\foreach \x in {3,...,\skipaa} {	
	\pgfmathtruncatemacro{\xx}{\x-2};
	\node[rotate=-90] at (\x-.5,-\n) [right] {\scriptsize $a_{1,n-\xx}$};
	}
\node[rotate=-90] at (2-.5,-\n) [right] {\scriptsize $a_{1n}$};
\node[rotate=-90] at (1-.6,-\n) [right] {\scriptsize $b_{1}$};
\node[rotate=-90] at (-1-.5,-\n) [right] {\scriptsize $c_{1n}$};
\foreach \x in {3,...,\skipaa} {	
	\pgfmathtruncatemacro{\xx}{\x-2};
	\node[rotate=-90] at (-\x+1-.5,-\n) [right] {\scriptsize $c_{1,n-\xx}$};
	}
\node[rotate=-90] at (-\n+1+.5,-\n) [right] {\scriptsize $c_{13}$};
}

\end{center}

\caption{Regular orbit of $\rho$}
\label{figure:regular_orbit}
\end{figure}
In each square or hexagon, parallel arrows have the same label, which stands for the reflection in the definition of the ``arrow'' relation. The picture appeared already in \cite{boe1985comparison}. The details can be found in \cite[4.3]{mrden2017phdthesis}. Recall again that this also gives the BGG resolution over $\iGr(2,2n)$ in the infinitesimal character $\rho$. Short arrows represent operators of order one, and longer arrows represent operators of order two. The resolution has the same shape in all other regular infinitesimal characters, but orders of the operators are higher in general. 

Plugging in a singular weight $\lambda+ \rho$ instead of a regular one, not all elements of its $W^\mfrp$-orbit need to be strictly $\mfrl$-dominant. Since such weights do not define a homogeneous sheaf or a generalized Verma module, we ignore them. We consider only the strictly $\mfrl$-dominant part of the $W^\mfrp$-orbit of $\lambda+ \rho$, and call it the singular orbit of the singular weight $\lambda+ \rho$. Some of the standard operators in the singular orbit become zero maps, and some become identities.

The purpose of this paper is to construct invariant differential operators that connect some of the connected components of the singular orbits, and to prove that certain sequences that include our constructed operators are exact. We will do this under the hypothesis that our infinitesimal character is semi-regular, and not orthogonal to the long simple root.
\begin{remark}
\label{remark:about_minimality}
Minimal semi-regular infinitesimal characters not orthogonal to the long simple root are given in (\ref{definition:lambda_k}). We will work only with these minimal ones, but the reader can notice that for our calculations the minimality is not necessary; only the ordering among the coordinates plays a role, and not their actual values. Alternatively, one can apply Jantzen-Zuckerman translation functors (see \cite[3.7]{baston1985algebraic}) to obtain the result in non-minimal infinitesimal characters from the result in the minimal case.
\end{remark}

For visual purpose, we put every strictly $\mfrl$-dominant weight $w(\lambda+\rho)$ in the coordinate system, with the same coordinates as $w \in W^\mfrp$ had in Figure \ref{figure:regular_orbit}. Note that in this coordinate system, the coordinates of $w(\lambda+\rho)$ do not have to correspond to its first two coordinates in $\mfrh^\ast$. But whenever we refer to an element of the singular orbit (or a homogeneous sheaf, or a generalized Verma module) as an ordered pair $(\mu_1,\mu_2)$, if not mentioned otherwise, we always mean the first two $\mfrh^\ast$-coordinates $(\mu_1,\mu_2 \pbar \ldots)$, and the rest of the coordinates are uniquely determined if an infinitesimal character is fixed.

\begin{proposition}
The singular orbits of $\lambda_k + \rho$ from (\ref{definition:lambda_k}) and (\ref{definition:lambda_0}) are given in Figure \ref{figure:singular_orbit_n_1} ($k=n-1$), Figure \ref{figure:singular_orbit_k} ($n-1>k>1$), Figure \ref{figure:singular_orbit_1} ($k=1$) and Figure \ref{figure:singular_orbit_0} ($k=0$). In each of the figures, a point with coordinates $(x_1,x_2)$ represents the weight $(\mu_1,\mu_2)$, where
\begin{equation}
\label{equation:singular_orbit_coordinates}
\mu_i =  \begin{cases}
    x_i & \colon |x_i| \leq k \\
    x_i - \sign(x_i) & \colon |x_i| > k
     \end{cases}.
\end{equation}
\end{proposition}
\begin{figure}
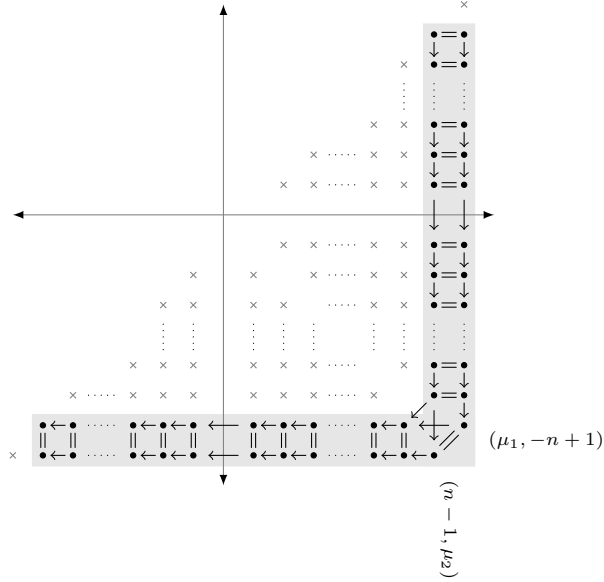

\begin{center}
\tikz[x=.4cm,y=.4cm]{ 

\def\n{8} 
\def\skipa{4} 
\def\skipb{11} 
\def\k{7} 

\pgfmathtruncatemacro{\kk}{\k + 1}
\pgfmathtruncatemacro{\kkk}{\kk + 1}

\def\margin{.35};
\draw[draw=black!10,fill=black!10] (\k-\margin,-\n-\margin) rectangle ++(1+2*\margin,\n+\k-1+2*\margin);
\draw[draw=black!10,fill=black!10] (-\k+1-\margin,-\k-1-\margin) rectangle ++(\n+\k-1+2*\margin,1+2*\margin);
\ifthenelse{\kk=\n}{
	\node at (\n+.5,-\k-.5) [right] {\scriptsize $(\mu_1,-n+1)$};
	\node[rotate=-90] at (\k+.5,-\n-.5) [right] {\scriptsize $(n-1,\mu_2)$};}{
	\draw[draw=black!10,fill=black!10] (\k+2-\margin,\k-\margin) rectangle ++(\n-\k-2+2*\margin,1+2*\margin);
	\draw[draw=black!10,fill=black!10] (-\k-1-\margin,-\n-\margin) rectangle ++(1+2*\margin,\n-\k-2+2*\margin);
	\node at (\n+.5,\k+.5) [right] {\scriptsize $(\mu_1,k)$};
	\node[rotate=-90] at (-\k-.5,-\n-.5) [right] {\scriptsize $(-k,\mu_2)$};
	\node at (\n+.5,-\k-.5) [right] {\scriptsize $(\mu_1,-k)$};
	\node[rotate=-90] at (\k+.5,-\n-.5) [right] {\scriptsize $(k,\mu_2)$};
}

\draw[latex-latex, thin, draw=gray] (-\n+1,0)--(\n+1,0) node [right] {}; 
\draw[latex-latex, thin, draw=gray] (0,-\n-1)--(0,\n-1) node [above] {}; 

\newcommand{\ldominant}[2]{\draw[fill] (#1,#2) circle [radius=1pt];} 

\newcommand{\arrow}[4]{\draw [->] (.75*#1+.25*#3,.75*#2+.25*#4) -- (.25*#1+.75*#3,.25*#2+.75*#4);} 

\newcommand{\skipped}[4]{\draw [dotted,-] (.75*#1+.25*#3,.75*#2+.25*#4) -- (.25*#1+.75*#3,.25*#2+.75*#4);} 

\def\lengthcross{.1}
\newcommand{\trivial}[2]{\draw[black!50] (#1-\lengthcross,#2-\lengthcross)--(#1+\lengthcross,#2+\lengthcross); \draw[black!50] (#1-\lengthcross,#2+\lengthcross)--(#1+\lengthcross,#2-\lengthcross);} 

\def\equaldis{.07}
\newcommand{\equal}[4]{
	\draw (.75*#1+.25*#3+\equaldis*#2-\equaldis*#4,.75*#2+.25*#4+\equaldis*#3-\equaldis*#1) -- (.25*#1+.75*#3+\equaldis*#2-\equaldis*#4,.25*#2+.75*#4+\equaldis*#3-\equaldis*#1);
	\draw (.75*#1+.25*#3-\equaldis*#2+\equaldis*#4,.75*#2+.25*#4-\equaldis*#3+\equaldis*#1) -- (.25*#1+.75*#3-\equaldis*#2+\equaldis*#4,.25*#2+.75*#4-\equaldis*#3+\equaldis*#1);
} 

\foreach \x in {1,...,\n} \foreach \y in {1,...,\n} {
	\ifthenelse{\not \x=\skipa \and \not \x=\skipb \and \not \y=\skipa \and \not \y=\skipb \and \not \x=\k \and \not \x=\kk \and \not \y=\k \and \not \y=\kk}{
		\ifthenelse{\x>\y}{\trivial{\x}{\y};}{}; 
		\ifthenelse{\x=\y}{}{\trivial{\x}{-\y};}; 
		\ifthenelse{\x<\y}{\trivial{-\x}{-\y};}{}; 
		}{}
	}
\trivial{\kk}{\k}; \trivial{-\k}{-\kk};

\ifthenelse{\kk<\n}{\foreach \x in {\kkk,...,\n} {
	\ifthenelse{\not \x=\skipb}{\ldominant{\x}{\k}; \ldominant{\x}{\kk};}{}
	}}{}
\foreach \x in {-\n,...,\k} {
	\ifthenelse{\not \x=\k \and \not \x=\kk \and \not \x=\skipa \and \not \x=0 \and \not \x=-\skipa \and \not \x=-\k \and \not \x=-\kk \and \not \x=-\skipb}{\ldominant{\k}{\x}; \ldominant{\kk}{\x};}{}
	}
\ifthenelse{\kk<\n}{\foreach \x in {-\n,...,-\kkk} {
	\ifthenelse{\not \x=-\skipb}{\ldominant{-\k}{\x}; \ldominant{-\kk}{\x};}{}
	}}{}
\foreach \x in {-\k,...,\n} {
	\ifthenelse{\not \x=\k \and \not \x=\kk \and \not \x=\skipa \and \not \x=\skipb \and \not \x=-\skipb \and \not \x=0 \and \not \x=-\skipa \and \not \x=-\k \and \not \x=-\kk}{\ldominant{\x}{-\k}; \ldominant{\x}{-\kk};}{}
	}
\ldominant{\k}{-\kk}; \ldominant{\kk}{-\k};

\arrow{7}{-5}{7}{-6};\arrow{7}{-2}{7}{-3};\arrow{7}{-1}{7}{-2};\arrow{7}{2}{7}{1};\arrow{7}{3}{7}{2};\arrow{7}{6}{7}{5};\arrow{8}{-6}{8}{-7};\arrow{8}{-5}{8}{-6};\arrow{8}{-2}{8}{-3};\arrow{8}{-1}{8}{-2};\arrow{8}{2}{8}{1};\arrow{8}{3}{8}{2};\arrow{8}{6}{8}{5};\arrow{-5}{-7}{-6}{-7};\arrow{-2}{-7}{-3}{-7};\arrow{-1}{-7}{-2}{-7};\arrow{2}{-7}{1}{-7};\arrow{3}{-7}{2}{-7};\arrow{6}{-7}{5}{-7};\arrow{-5}{-8}{-6}{-8};\arrow{-2}{-8}{-3}{-8};\arrow{-1}{-8}{-2}{-8};\arrow{2}{-8}{1}{-8};\arrow{3}{-8}{2}{-8};\arrow{6}{-8}{5}{-8};\arrow{7}{-8}{6}{-8};\arrow{7}{1}{7}{-1};\arrow{8}{1}{8}{-1};\arrow{7}{-6}{7}{-8};\arrow{8}{-7}{6}{-7};\arrow{7}{-6}{6}{-7};\arrow{1}{-7}{-1}{-7};\arrow{1}{-8}{-1}{-8};
\skipped{3}{-8}{5}{-8};\skipped{3}{-7}{5}{-7};\skipped{3}{-6}{5}{-6};\skipped{3}{-5}{5}{-5};\skipped{3}{-3}{5}{-3};\skipped{3}{-2}{5}{-2};\skipped{3}{-1}{5}{-1};\skipped{3}{1}{5}{1};\skipped{3}{2}{5}{2};\skipped{-5}{-8}{-3}{-8};\skipped{-5}{-7}{-3}{-7};\skipped{-5}{-6}{-3}{-6};\skipped{6}{3}{6}{5};\skipped{7}{3}{7}{5};\skipped{8}{3}{8}{5};\skipped{-2}{-5}{-2}{-3};\skipped{-1}{-5}{-1}{-3};\skipped{1}{-5}{1}{-3};\skipped{2}{-5}{2}{-3};\skipped{3}{-5}{3}{-3};\skipped{5}{-5}{5}{-3};\skipped{6}{-5}{6}{-3};\skipped{7}{-5}{7}{-3};\skipped{8}{-5}{8}{-3};
\equal{8}{-7}{7}{-8};\equal{7}{-6}{8}{-6};\equal{7}{-5}{8}{-5};\equal{7}{-3}{8}{-3};\equal{7}{-2}{8}{-2};\equal{7}{-1}{8}{-1};\equal{7}{1}{8}{1};\equal{7}{2}{8}{2};\equal{7}{3}{8}{3};\equal{7}{5}{8}{5};\equal{7}{6}{8}{6};\equal{-6}{-7}{-6}{-8};\equal{-5}{-7}{-5}{-8};\equal{-3}{-7}{-3}{-8};\equal{-2}{-7}{-2}{-8};\equal{-1}{-7}{-1}{-8};\equal{1}{-7}{1}{-8};\equal{2}{-7}{2}{-8};\equal{3}{-7}{3}{-8};\equal{5}{-7}{5}{-8};\equal{6}{-7}{6}{-8};

}
\end{center}

\caption{Singular orbit of $\lambda_{n-1} + \rho$}
\label{figure:singular_orbit_n_1}
\end{figure}
\begin{figure}
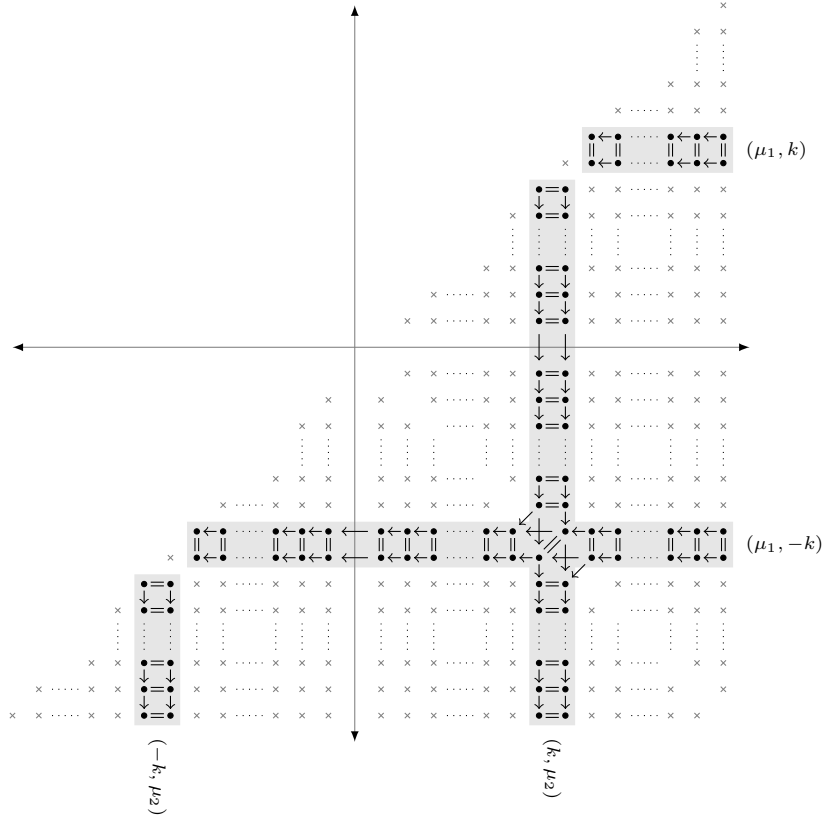

\begin{center}
\tikz[x=.35cm,y=.35cm]{ 

\def\n{14} 
\def\skipa{4} 
\def\skipb{11} 
\def\k{7} 

\pgfmathtruncatemacro{\kk}{\k + 1}
\pgfmathtruncatemacro{\kkk}{\kk + 1}

\def\margin{.35};
\draw[draw=black!10,fill=black!10] (\k-\margin,-\n-\margin) rectangle ++(1+2*\margin,\n+\k-1+2*\margin);
\draw[draw=black!10,fill=black!10] (-\k+1-\margin,-\k-1-\margin) rectangle ++(\n+\k-1+2*\margin,1+2*\margin);
\ifthenelse{\kk=\n}{
	\node at (\n+.5,-\k-.5) [right] {\scriptsize $(\mu_1,-n+1)$};
	\node[rotate=-90] at (\k+.5,-\n-.5) [right] {\scriptsize $(n-1,\mu_2)$};}{
	\draw[draw=black!10,fill=black!10] (\k+2-\margin,\k-\margin) rectangle ++(\n-\k-2+2*\margin,1+2*\margin);
	\draw[draw=black!10,fill=black!10] (-\k-1-\margin,-\n-\margin) rectangle ++(1+2*\margin,\n-\k-2+2*\margin);
	\node at (\n+.5,\k+.5) [right] {\scriptsize $(\mu_1,k)$};
	\node[rotate=-90] at (-\k-.5,-\n-.5) [right] {\scriptsize $(-k,\mu_2)$};
	\node at (\n+.5,-\k-.5) [right] {\scriptsize $(\mu_1,-k)$};
	\node[rotate=-90] at (\k+.5,-\n-.5) [right] {\scriptsize $(k,\mu_2)$};
}

\draw[latex-latex, thin, draw=gray] (-\n+1,0)--(\n+1,0) node [right] {}; 
\draw[latex-latex, thin, draw=gray] (0,-\n-1)--(0,\n-1) node [above] {}; 

\newcommand{\ldominant}[2]{\draw[fill] (#1,#2) circle [radius=1pt];} 

\newcommand{\arrow}[4]{\draw [->] (.75*#1+.25*#3,.75*#2+.25*#4) -- (.25*#1+.75*#3,.25*#2+.75*#4);} 

\newcommand{\skipped}[4]{\draw [dotted,-] (.75*#1+.25*#3,.75*#2+.25*#4) -- (.25*#1+.75*#3,.25*#2+.75*#4);} 

\def\lengthcross{.1}
\newcommand{\trivial}[2]{\draw[black!50] (#1-\lengthcross,#2-\lengthcross)--(#1+\lengthcross,#2+\lengthcross); \draw[black!50] (#1-\lengthcross,#2+\lengthcross)--(#1+\lengthcross,#2-\lengthcross);} 

\def\equaldis{.07}
\newcommand{\equal}[4]{
	\draw (.75*#1+.25*#3+\equaldis*#2-\equaldis*#4,.75*#2+.25*#4+\equaldis*#3-\equaldis*#1) -- (.25*#1+.75*#3+\equaldis*#2-\equaldis*#4,.25*#2+.75*#4+\equaldis*#3-\equaldis*#1);
	\draw (.75*#1+.25*#3-\equaldis*#2+\equaldis*#4,.75*#2+.25*#4-\equaldis*#3+\equaldis*#1) -- (.25*#1+.75*#3-\equaldis*#2+\equaldis*#4,.25*#2+.75*#4-\equaldis*#3+\equaldis*#1);
} 

\foreach \x in {1,...,\n} \foreach \y in {1,...,\n} {
	\ifthenelse{\not \x=\skipa \and \not \x=\skipb \and \not \y=\skipa \and \not \y=\skipb \and \not \x=\k \and \not \x=\kk \and \not \y=\k \and \not \y=\kk}{
		\ifthenelse{\x>\y}{\trivial{\x}{\y};}{}; 
		\ifthenelse{\x=\y}{}{\trivial{\x}{-\y};}; 
		\ifthenelse{\x<\y}{\trivial{-\x}{-\y};}{}; 
		}{}
	}
\trivial{\kk}{\k}; \trivial{-\k}{-\kk};

\ifthenelse{\kk<\n}{\foreach \x in {\kkk,...,\n} {
	\ifthenelse{\not \x=\skipb}{\ldominant{\x}{\k}; \ldominant{\x}{\kk};}{}
	}}{}
\foreach \x in {-\n,...,\k} {
	\ifthenelse{\not \x=\k \and \not \x=\kk \and \not \x=\skipa \and \not \x=0 \and \not \x=-\skipa \and \not \x=-\k \and \not \x=-\kk \and \not \x=-\skipb}{\ldominant{\k}{\x}; \ldominant{\kk}{\x};}{}
	}
\ifthenelse{\kk<\n}{\foreach \x in {-\n,...,-\kkk} {
	\ifthenelse{\not \x=-\skipb}{\ldominant{-\k}{\x}; \ldominant{-\kk}{\x};}{}
	}}{}
\foreach \x in {-\k,...,\n} {
	\ifthenelse{\not \x=\k \and \not \x=\kk \and \not \x=\skipa \and \not \x=\skipb \and \not \x=-\skipb \and \not \x=0 \and \not \x=-\skipa \and \not \x=-\k \and \not \x=-\kk}{\ldominant{\x}{-\k}; \ldominant{\x}{-\kk};}{}
	}
\ldominant{\k}{-\kk}; \ldominant{\kk}{-\k};

\arrow{10}{7}{9}{7};\arrow{13}{7}{12}{7};\arrow{14}{7}{13}{7};\arrow{10}{8}{9}{8};\arrow{13}{8}{12}{8};\arrow{14}{8}{13}{8};\arrow{7}{-13}{7}{-14};\arrow{7}{-12}{7}{-13};\arrow{7}{-9}{7}{-10};\arrow{7}{-8}{7}{-9};\arrow{7}{-5}{7}{-6};\arrow{7}{-2}{7}{-3};\arrow{7}{-1}{7}{-2};\arrow{7}{2}{7}{1};\arrow{7}{3}{7}{2};\arrow{7}{6}{7}{5};\arrow{8}{-13}{8}{-14};\arrow{8}{-12}{8}{-13};\arrow{8}{-9}{8}{-10};\arrow{8}{-6}{8}{-7};\arrow{8}{-5}{8}{-6};\arrow{8}{-2}{8}{-3};\arrow{8}{-1}{8}{-2};\arrow{8}{2}{8}{1};\arrow{8}{3}{8}{2};\arrow{8}{6}{8}{5};\arrow{-5}{-7}{-6}{-7};\arrow{-2}{-7}{-3}{-7};\arrow{-1}{-7}{-2}{-7};\arrow{2}{-7}{1}{-7};\arrow{3}{-7}{2}{-7};\arrow{6}{-7}{5}{-7};\arrow{9}{-7}{8}{-7};\arrow{10}{-7}{9}{-7};\arrow{13}{-7}{12}{-7};\arrow{14}{-7}{13}{-7};\arrow{-5}{-8}{-6}{-8};\arrow{-2}{-8}{-3}{-8};\arrow{-1}{-8}{-2}{-8};\arrow{2}{-8}{1}{-8};\arrow{3}{-8}{2}{-8};\arrow{6}{-8}{5}{-8};\arrow{7}{-8}{6}{-8};\arrow{10}{-8}{9}{-8};\arrow{13}{-8}{12}{-8};\arrow{14}{-8}{13}{-8};\arrow{7}{1}{7}{-1};\arrow{8}{1}{8}{-1};\arrow{7}{-6}{7}{-8};\arrow{8}{-7}{6}{-7};\arrow{7}{-6}{6}{-7};\arrow{1}{-7}{-1}{-7};\arrow{1}{-8}{-1}{-8};\arrow{9}{-8}{8}{-9};\arrow{8}{-7}{8}{-9};\arrow{9}{-8}{7}{-8};\arrow{-7}{-13}{-7}{-14};\arrow{-7}{-12}{-7}{-13};\arrow{-7}{-9}{-7}{-10};\arrow{-8}{-13}{-8}{-14};\arrow{-8}{-12}{-8}{-13};\arrow{-8}{-9}{-8}{-10};
\skipped{10}{-14}{12}{-14};\skipped{10}{-13}{12}{-13};\skipped{10}{-12}{12}{-12};\skipped{10}{-10}{12}{-10};\skipped{10}{-9}{12}{-9};\skipped{10}{-8}{12}{-8};\skipped{10}{-7}{12}{-7};\skipped{10}{-6}{12}{-6};\skipped{10}{-5}{12}{-5};\skipped{10}{-3}{12}{-3};\skipped{10}{-2}{12}{-2};\skipped{10}{-1}{12}{-1};\skipped{10}{1}{12}{1};\skipped{10}{2}{12}{2};\skipped{10}{3}{12}{3};\skipped{10}{5}{12}{5};\skipped{10}{6}{12}{6};\skipped{10}{7}{12}{7};\skipped{10}{8}{12}{8};\skipped{10}{9}{12}{9};\skipped{3}{-14}{5}{-14};\skipped{3}{-13}{5}{-13};\skipped{3}{-12}{5}{-12};\skipped{3}{-10}{5}{-10};\skipped{3}{-9}{5}{-9};\skipped{3}{-8}{5}{-8};\skipped{3}{-7}{5}{-7};\skipped{3}{-6}{5}{-6};\skipped{3}{-5}{5}{-5};\skipped{3}{-3}{5}{-3};\skipped{3}{-2}{5}{-2};\skipped{3}{-1}{5}{-1};\skipped{3}{1}{5}{1};\skipped{3}{2}{5}{2};\skipped{-5}{-14}{-3}{-14};\skipped{-5}{-13}{-3}{-13};\skipped{-5}{-12}{-3}{-12};\skipped{-5}{-10}{-3}{-10};\skipped{-5}{-9}{-3}{-9};\skipped{-5}{-8}{-3}{-8};\skipped{-5}{-7}{-3}{-7};\skipped{-5}{-6}{-3}{-6};\skipped{-12}{-14}{-10}{-14};\skipped{-12}{-13}{-10}{-13};\skipped{13}{10}{13}{12};\skipped{14}{10}{14}{12};\skipped{6}{3}{6}{5};\skipped{7}{3}{7}{5};\skipped{8}{3}{8}{5};\skipped{9}{3}{9}{5};\skipped{10}{3}{10}{5};\skipped{12}{3}{12}{5};\skipped{13}{3}{13}{5};\skipped{14}{3}{14}{5};\skipped{-2}{-5}{-2}{-3};\skipped{-1}{-5}{-1}{-3};\skipped{1}{-5}{1}{-3};\skipped{2}{-5}{2}{-3};\skipped{3}{-5}{3}{-3};\skipped{5}{-5}{5}{-3};\skipped{6}{-5}{6}{-3};\skipped{7}{-5}{7}{-3};\skipped{8}{-5}{8}{-3};\skipped{9}{-5}{9}{-3};\skipped{10}{-5}{10}{-3};\skipped{12}{-5}{12}{-3};\skipped{13}{-5}{13}{-3};\skipped{14}{-5}{14}{-3};\skipped{-9}{-12}{-9}{-10};\skipped{-8}{-12}{-8}{-10};\skipped{-7}{-12}{-7}{-10};\skipped{-6}{-12}{-6}{-10};\skipped{-5}{-12}{-5}{-10};\skipped{-3}{-12}{-3}{-10};\skipped{-2}{-12}{-2}{-10};\skipped{-1}{-12}{-1}{-10};\skipped{1}{-12}{1}{-10};\skipped{2}{-12}{2}{-10};\skipped{3}{-12}{3}{-10};\skipped{5}{-12}{5}{-10};\skipped{6}{-12}{6}{-10};\skipped{7}{-12}{7}{-10};\skipped{8}{-12}{8}{-10};\skipped{9}{-12}{9}{-10};\skipped{10}{-12}{10}{-10};\skipped{12}{-12}{12}{-10};\skipped{13}{-12}{13}{-10};\skipped{14}{-12}{14}{-10};
\equal{8}{-7}{7}{-8};\equal{9}{7}{9}{8};\equal{10}{7}{10}{8};\equal{12}{7}{12}{8};\equal{13}{7}{13}{8};\equal{14}{7}{14}{8};\equal{7}{-14}{8}{-14};\equal{7}{-13}{8}{-13};\equal{7}{-12}{8}{-12};\equal{7}{-10}{8}{-10};\equal{7}{-9}{8}{-9};\equal{7}{-6}{8}{-6};\equal{7}{-5}{8}{-5};\equal{7}{-3}{8}{-3};\equal{7}{-2}{8}{-2};\equal{7}{-1}{8}{-1};\equal{7}{1}{8}{1};\equal{7}{2}{8}{2};\equal{7}{3}{8}{3};\equal{7}{5}{8}{5};\equal{7}{6}{8}{6};\equal{-6}{-7}{-6}{-8};\equal{-5}{-7}{-5}{-8};\equal{-3}{-7}{-3}{-8};\equal{-2}{-7}{-2}{-8};\equal{-1}{-7}{-1}{-8};\equal{1}{-7}{1}{-8};\equal{2}{-7}{2}{-8};\equal{3}{-7}{3}{-8};\equal{5}{-7}{5}{-8};\equal{6}{-7}{6}{-8};\equal{9}{-7}{9}{-8};\equal{10}{-7}{10}{-8};\equal{12}{-7}{12}{-8};\equal{13}{-7}{13}{-8};\equal{14}{-7}{14}{-8};\equal{-7}{-14}{-8}{-14};\equal{-7}{-13}{-8}{-13};\equal{-7}{-12}{-8}{-12};\equal{-7}{-10}{-8}{-10};\equal{-7}{-9}{-8}{-9};

}

\end{center}

\caption{Singular orbit of $\lambda_k + \rho$ for $n-1>k>1$}
\label{figure:singular_orbit_k}
\end{figure}
\begin{figure}
\begin{center}
\tikz[x=.4cm,y=.4cm]{ 

\def\n{8} 
\def\skipa{5} 
\def\skipb{11} 
\def\k{1} 

\pgfmathtruncatemacro{\kk}{\k + 1}
\pgfmathtruncatemacro{\kkk}{\kk + 1}

\def\margin{.35};
\draw[draw=black!10,fill=black!10] (\k-\margin,-\n-\margin) rectangle ++(1+2*\margin,\n+\k-2+2*\margin);
\draw[draw=black!10,fill=black!10] (\k-\margin,-\k-1-\margin) rectangle ++(\n+\k-2+2*\margin,1+2*\margin);
\draw[draw=black!10,fill=black!10] (\k+2-\margin,\k-\margin) rectangle ++(\n-\k-2+2*\margin,1+2*\margin);
\draw[draw=black!10,fill=black!10] (-\k-1-\margin,-\n-\margin) rectangle ++(1+2*\margin,\n-\k-2+2*\margin);
\node at (\n+.5,\k+.5) [right] {\scriptsize $(\mu_1,1)$};
\node[rotate=-90] at (-\k-.5,-\n-.5) [right] {\scriptsize $(-1,\mu_2)$};
\node at (\n+.5,-\k-.5) [right] {\scriptsize $(\mu_1,-1)$};
\node[rotate=-90] at (\k+.5,-\n-.5) [right] {\scriptsize $(1,\mu_2)$};

\draw[latex-latex, thin, draw=gray] (-\n+1,0)--(\n+1,0) node [right] {}; 
\draw[latex-latex, thin, draw=gray] (0,-\n-1)--(0,\n-1) node [above] {}; 

\newcommand{\ldominant}[2]{\draw[fill] (#1,#2) circle [radius=1pt];} 

\newcommand{\arrow}[4]{\draw [->] (.75*#1+.25*#3,.75*#2+.25*#4) -- (.25*#1+.75*#3,.25*#2+.75*#4);} 

\newcommand{\skipped}[4]{\draw [dotted,-] (.75*#1+.25*#3,.75*#2+.25*#4) -- (.25*#1+.75*#3,.25*#2+.75*#4);} 

\def\lengthcross{.1}
\newcommand{\trivial}[2]{\draw[black!50] (#1-\lengthcross,#2-\lengthcross)--(#1+\lengthcross,#2+\lengthcross); \draw[black!50] (#1-\lengthcross,#2+\lengthcross)--(#1+\lengthcross,#2-\lengthcross);} 

\def\equaldis{.07}
\newcommand{\equal}[4]{
	\draw (.75*#1+.25*#3+\equaldis*#2-\equaldis*#4,.75*#2+.25*#4+\equaldis*#3-\equaldis*#1) -- (.25*#1+.75*#3+\equaldis*#2-\equaldis*#4,.25*#2+.75*#4+\equaldis*#3-\equaldis*#1);
	\draw (.75*#1+.25*#3-\equaldis*#2+\equaldis*#4,.75*#2+.25*#4-\equaldis*#3+\equaldis*#1) -- (.25*#1+.75*#3-\equaldis*#2+\equaldis*#4,.25*#2+.75*#4-\equaldis*#3+\equaldis*#1);
} 

\foreach \x in {1,...,\n} \foreach \y in {1,...,\n} {
	\ifthenelse{\not \x=\skipa \and \not \x=\skipb \and \not \y=\skipa \and \not \y=\skipb \and \not \x=\k \and \not \x=\kk \and \not \y=\k \and \not \y=\kk}{
		\ifthenelse{\x>\y}{\trivial{\x}{\y};}{}; 
		\ifthenelse{\x=\y}{}{\trivial{\x}{-\y};}; 
		\ifthenelse{\x<\y}{\trivial{-\x}{-\y};}{}; 
		}{}
	}
\trivial{\kk}{\k}; \trivial{-\k}{-\kk};

\ifthenelse{\kk<\n}{\foreach \x in {\kkk,...,\n} {
	\ifthenelse{\not \x=\skipb \and \not \x=\skipa}{\ldominant{\x}{\k}; \ldominant{\x}{\kk};}{}
	}}{}
\foreach \x in {-\n,...,\k} {
	\ifthenelse{\not \x=\k \and \not \x=\kk \and \not \x=\skipa \and \not \x=0 \and \not \x=-\skipa \and \not \x=-\k \and \not \x=-\kk \and \not \x=-\skipb}{\ldominant{\k}{\x}; \ldominant{\kk}{\x};}{}
	}
\ifthenelse{\kk<\n}{\foreach \x in {-\n,...,-\kkk} {
	\ifthenelse{\not \x=-\skipb \and \not \x=-\skipa}{\ldominant{-\k}{\x}; \ldominant{-\kk}{\x};}{}
	}}{}
\foreach \x in {-\k,...,\n} {
	\ifthenelse{\not \x=\k \and \not \x=\kk \and \not \x=\skipa \and \not \x=\skipb \and \not \x=-\skipb \and \not \x=0 \and \not \x=-\skipa \and \not \x=-\k \and \not \x=-\kk}{\ldominant{\x}{-\k}; \ldominant{\x}{-\kk};}{}
	}
\ldominant{\k}{-\kk}; \ldominant{\kk}{-\k};

\arrow{4}{1}{3}{1};\arrow{7}{1}{6}{1};\arrow{8}{1}{7}{1};\arrow{4}{2}{3}{2};\arrow{7}{2}{6}{2};\arrow{8}{2}{7}{2};\arrow{1}{-7}{1}{-8};\arrow{1}{-6}{1}{-7};\arrow{1}{-3}{1}{-4};\arrow{1}{-2}{1}{-3};\arrow{2}{-7}{2}{-8};\arrow{2}{-6}{2}{-7};\arrow{2}{-3}{2}{-4};\arrow{3}{-1}{2}{-1};\arrow{4}{-1}{3}{-1};\arrow{7}{-1}{6}{-1};\arrow{8}{-1}{7}{-1};\arrow{4}{-2}{3}{-2};\arrow{7}{-2}{6}{-2};\arrow{8}{-2}{7}{-2};\arrow{3}{-2}{2}{-3};\arrow{2}{-1}{2}{-3};\arrow{3}{-2}{1}{-2};\arrow{-1}{-7}{-1}{-8};\arrow{-1}{-6}{-1}{-7};\arrow{-1}{-3}{-1}{-4};\arrow{-2}{-7}{-2}{-8};\arrow{-2}{-6}{-2}{-7};\arrow{-2}{-3}{-2}{-4};
\skipped{4}{-8}{6}{-8};\skipped{4}{-7}{6}{-7};\skipped{4}{-6}{6}{-6};\skipped{4}{-4}{6}{-4};\skipped{4}{-3}{6}{-3};\skipped{4}{-2}{6}{-2};\skipped{4}{-1}{6}{-1};\skipped{4}{1}{6}{1};\skipped{4}{2}{6}{2};\skipped{4}{3}{6}{3};\skipped{-6}{-8}{-4}{-8};\skipped{-6}{-7}{-4}{-7};\skipped{7}{4}{7}{6};\skipped{8}{4}{8}{6};\skipped{-3}{-6}{-3}{-4};\skipped{-2}{-6}{-2}{-4};\skipped{-1}{-6}{-1}{-4};\skipped{1}{-6}{1}{-4};\skipped{2}{-6}{2}{-4};\skipped{3}{-6}{3}{-4};\skipped{4}{-6}{4}{-4};\skipped{6}{-6}{6}{-4};\skipped{7}{-6}{7}{-4};\skipped{8}{-6}{8}{-4};
\equal{2}{-1}{1}{-2};\equal{3}{1}{3}{2};\equal{4}{1}{4}{2};\equal{6}{1}{6}{2};\equal{7}{1}{7}{2};\equal{8}{1}{8}{2};\equal{1}{-8}{2}{-8};\equal{1}{-7}{2}{-7};\equal{1}{-6}{2}{-6};\equal{1}{-4}{2}{-4};\equal{1}{-3}{2}{-3};\equal{3}{-1}{3}{-2};\equal{4}{-1}{4}{-2};\equal{6}{-1}{6}{-2};\equal{7}{-1}{7}{-2};\equal{8}{-1}{8}{-2};\equal{-1}{-8}{-2}{-8};\equal{-1}{-7}{-2}{-7};\equal{-1}{-6}{-2}{-6};\equal{-1}{-4}{-2}{-4};\equal{-1}{-3}{-2}{-3};

}

\end{center}

\caption{Singular orbit of $\lambda_{1} + \rho$}
\label{figure:singular_orbit_1}
\end{figure}
\begin{figure}
\begin{center}
\tikz[x=.4cm,y=.4cm]{ 

\def\n{8} 
\def\skipa{5} 
\def\skipb{11} 
\def\k{0} 

\pgfmathtruncatemacro{\kk}{\k + 1}
\pgfmathtruncatemacro{\kkk}{\kk + 1}

\def\margin{.35};
\draw[draw=black!10,fill=black!10] (\k-1-\margin,-\n-\margin) rectangle ++(2+2*\margin,\n+\k-2+2*\margin);

\draw[draw=black!10,fill=black!10] (\k+2-\margin,\k-1-\margin) rectangle ++(\n-\k-2+2*\margin,2+2*\margin);
\draw[draw=black!10,fill=black!10] (-\k-1-\margin,-\n-\margin) rectangle ++(1+2*\margin,\n-\k-2+2*\margin);

\node at (\n+1,0) [right] {\scriptsize $(\mu_1,0)$};
\node[rotate=-90] at (0,-\n-1) [right] {\scriptsize $(0,\mu_2)$};
\draw[latex-latex, thin, draw=gray] (-\n+1,0)--(\n+1,0) node [right] {}; 
\draw[latex-latex, thin, draw=gray] (0,-\n-1)--(0,\n-1) node [above] {}; 

\newcommand{\ldominant}[2]{\draw[fill] (#1,#2) circle [radius=1pt];} 

\newcommand{\arrow}[4]{\draw [->] (.75*#1+.25*#3,.75*#2+.25*#4) -- (.25*#1+.75*#3,.25*#2+.75*#4);} 

\newcommand{\skipped}[4]{\draw [dotted,-] (.75*#1+.25*#3,.75*#2+.25*#4) -- (.25*#1+.75*#3,.25*#2+.75*#4);} 

\def\lengthcross{.1}
\newcommand{\trivial}[2]{\draw[black!50] (#1-\lengthcross,#2-\lengthcross)--(#1+\lengthcross,#2+\lengthcross); \draw[black!50] (#1-\lengthcross,#2+\lengthcross)--(#1+\lengthcross,#2-\lengthcross);} 

\def\equaldis{.07}
\newcommand{\equal}[4]{
	\draw (.75*#1+.25*#3+\equaldis*#2-\equaldis*#4,.75*#2+.25*#4+\equaldis*#3-\equaldis*#1) -- (.25*#1+.75*#3+\equaldis*#2-\equaldis*#4,.25*#2+.75*#4+\equaldis*#3-\equaldis*#1);
	\draw (.75*#1+.25*#3-\equaldis*#2+\equaldis*#4,.75*#2+.25*#4-\equaldis*#3+\equaldis*#1) -- (.25*#1+.75*#3-\equaldis*#2+\equaldis*#4,.25*#2+.75*#4-\equaldis*#3+\equaldis*#1);
} 

\foreach \x in {1,...,\n} \foreach \y in {1,...,\n} {
	\ifthenelse{\not \x=\skipa \and \not \x=\skipb \and \not \y=\skipa \and \not \y=\skipb \and \not \x=\k \and \not \x=\kk \and \not \y=\k \and \not \y=\kk}{
		\ifthenelse{\x>\y}{\trivial{\x}{\y};}{}; 
		\ifthenelse{\x=\y}{}{\trivial{\x}{-\y};}; 
		\ifthenelse{\x<\y}{\trivial{-\x}{-\y};}{}; 
		}{}
	}

\ifthenelse{\kk<\n}{\foreach \x in {\kkk,...,\n} {
	\ifthenelse{\not \x=\skipb \and \not \x=\skipa}{\ldominant{\x}{\kk};}{}
	}}{}
\foreach \x in {-\n,...,\k} {
	\ifthenelse{\not \x=\k \and \not \x=\kk \and \not \x=\skipa \and \not \x=0 \and \not \x=-\skipa \and \not \x=-\k \and \not \x=-\kk \and \not \x=-\skipb}{\ldominant{\kk}{\x};}{}
	}
\ifthenelse{\kk<\n}{\foreach \x in {-\n,...,-\kkk} {
	\ifthenelse{\not \x=-\skipb \and \not \x=-\skipa}{\ldominant{-\kk}{\x};}{}
	}}{}
\foreach \x in {-\k,...,\n} {
	\ifthenelse{\not \x=\k \and \not \x=\kk \and \not \x=\skipa \and \not \x=\skipb \and \not \x=-\skipb \and \not \x=0 \and \not \x=-\skipa \and \not \x=-\k \and \not \x=-\kk}{\ldominant{\x}{-\kk};}{}
	}

\arrow{3}{1}{2}{1};\arrow{4}{1}{3}{1};\arrow{7}{1}{6}{1};\arrow{8}{1}{7}{1};\arrow{1}{-7}{1}{-8};\arrow{1}{-6}{1}{-7};\arrow{1}{-3}{1}{-4};\arrow{1}{-2}{1}{-3};\arrow{3}{-1}{2}{-1};\arrow{4}{-1}{3}{-1};\arrow{7}{-1}{6}{-1};\arrow{8}{-1}{7}{-1};\arrow{-1}{-7}{-1}{-8};\arrow{-1}{-6}{-1}{-7};\arrow{-1}{-3}{-1}{-4};\arrow{-1}{-2}{-1}{-3};
\skipped{4}{-8}{6}{-8};\skipped{4}{-7}{6}{-7};\skipped{4}{-6}{6}{-6};\skipped{4}{-4}{6}{-4};\skipped{4}{-3}{6}{-3};\skipped{4}{-2}{6}{-2};\skipped{4}{-1}{6}{-1};\skipped{4}{1}{6}{1};\skipped{4}{2}{6}{2};\skipped{4}{3}{6}{3};\skipped{-6}{-8}{-4}{-8};\skipped{-6}{-7}{-4}{-7};\skipped{7}{4}{7}{6};\skipped{8}{4}{8}{6};\skipped{-3}{-6}{-3}{-4};\skipped{-2}{-6}{-2}{-4};\skipped{-1}{-6}{-1}{-4};\skipped{1}{-6}{1}{-4};\skipped{2}{-6}{2}{-4};\skipped{3}{-6}{3}{-4};\skipped{4}{-6}{4}{-4};\skipped{6}{-6}{6}{-4};\skipped{7}{-6}{7}{-4};\skipped{8}{-6}{8}{-4};
\equal{2}{1}{2}{-1};\equal{3}{1}{3}{-1};\equal{4}{1}{4}{-1};\equal{6}{1}{6}{-1};\equal{7}{1}{7}{-1};\equal{8}{1}{8}{-1};\equal{1}{-8}{-1}{-8};\equal{1}{-7}{-1}{-7};\equal{1}{-6}{-1}{-6};\equal{1}{-4}{-1}{-4};\equal{1}{-3}{-1}{-3};\equal{1}{-2}{-1}{-2};

}
\end{center}

\caption{Singular orbit of $\lambda_{0} + \rho$}
\label{figure:singular_orbit_0}
\end{figure}

\begin{proof} The formula (\ref{equation:singular_orbit_coordinates}) is clear. A strictly $\mfrl$-dominant weight in the orbit cannot have both coordinates $k$ on the right-hand side of the bar. If $k=0$, then it cannot be on the right-hand side of the bar. The rest is easy checking of all the possibilities, and finding these possibilities in Figure \ref{figure:regular_orbit}.

In Figure \ref{figure:singular_orbit_1}, in principle there are the standard operators $(\mu_1,1) \to (\mu_1,-1)$ and $(1,\mu_2) \to (-1,\mu_2)$. Also, in Figure \ref{figure:singular_orbit_0} there is the standard operator $(1,0) \to (-1,0)$. But all of them are actually trivial, due to \cite{baston1985algebraic} (Theorem C in 3.5 and Theorem A in 3.6), so these arrows are omitted.
\end{proof}

Note that if we choose for each pair of equal points in the orbit the one that is upper or to the right, and consider the corresponding Hasse diagram element, we get the singular Hasse diagram $^S W^J$ in the sense of \cite{boe2005representation}, where $S=\Pi\setminus \{\alpha_2\}$ and $J=\{\alpha_k\}$.

\section{Non-standard operators and singular BGG complexes}

The main result of the paper is the following theorem:
\begin{theorem}
\label{theorem:main}
Fix $1<k<n-1$, and consider the singular orbit of $\lambda_k + \rho$. There exist invariant differential operators
\begin{gather}
\label{equation:non_std_k+} (k+1,k) \to (k,k-1), \\
\label{equation:non_std_k-} (-k+1,-k) \to (-k,-k-1),
\end{gather}
which are non-standard and of order two. Moreover, the following sequences of homogeneous sheaves and invariant differential operators are exact (in positive degrees) over the big affine cell in $\iGr(2,2n)$:
\begin{multline}
\label{equation:sing_BGG_k+}
(n-1,k) \to (n-2,k) \to \ldots \to (k+1,k) \to (k,k-1) \to \\ \to (k,k-2) \to \ldots \to (k,-n+1) \to 0,
\end{multline}
\begin{multline}
\label{equation:sing_BGG_k-}
(n-1,-k) \to (n-2,-k) \to \ldots \to (-k+1,-k) \to (-k,-k-1) \to \\ \to (-k,-k-2) \to \ldots \to (-k,-n+1) \to 0.
\end{multline}
All other arrows above represent the standard operators.

Similarly, for $k=1$, there exist invariant differential operators
\begin{gather}
\label{equation:non_std_1+} (2,1) \to (1,-1), \\
\label{equation:non_std_1-} (1,-1) \to (-1,-2),
\end{gather}
which are non-standard and of order at most three. Moreover, the following sequences are exact over the big affine cell in $\iGr(2,2n)$:
\begin{multline}
\label{equation:sing_BGG_1+}
(n-1,1) \to (n-2,1) \to \ldots \to (2,1) \to (1,-1) \to \\ \to (1,-2) \to \ldots \to (1,-n+1) \to 0,
\end{multline}
\begin{multline}
\label{equation:sing_BGG_1-}
(n-1,-1) \to (n-2,-1) \to \ldots \to (1,-1) \to (-1,-2) \to \\ \to (-1,-3) \to \ldots \to (-1,-n+1) \to 0.
\end{multline}

Finally, for $k=n-1$, the following sequences are exact over the big affine cell in $\iGr(2,2n)$ (no non-standard operators are needed here):
\begin{multline}
\label{equation:sing_BGG_n_1+}
(n-1,n-2) \to (n-1,n-3) \to \ldots \to (n-1,1) \to (n-1,-1) \to \\ \to (n-1,-2) \to \ldots \to (n-1,-n+1) \to 0,
\end{multline}
\begin{multline}
\label{equation:sing_BGG_n_1-}
(n-1,-n+1) \to (n-2,-n+1) \to \ldots \to (1,-n+1) \to (-1,-n+1) \to \\ \to (-2,-n+1) \to \ldots \to (-n+2,-n+1) \to 0.
\end{multline}
\end{theorem}
The sequences in Theorem \ref{theorem:main} should be thought of as analogues of BGG resolutions in singular infinitesimal characters. They resolve the solution spaces of certain partial differential equations, the kernels of the first operator in each resolution. We will describe these operators later. These sequences will be referred to as the positive ((\ref{equation:sing_BGG_k+}), (\ref{equation:sing_BGG_1+}) and (\ref{equation:sing_BGG_n_1+})) and the negative((\ref{equation:sing_BGG_k-}), (\ref{equation:sing_BGG_1-}) and (\ref{equation:sing_BGG_n_1-})) singular BGG complexes. Note that the object $(k,-k)$ always appears in both the positive and negative BGG complex. The analogous results hold in non-minimal cases (see Remark \ref{remark:about_minimality}), but of course the orders of the differential operators are higher.

\begin{proof}
Consider the Penrose transform over the following double fibration:
\begin{equation}
\label{equation:double_fib_1_2}
\xymatrix@C=-4em@R=2em{ & G/Q = \begin{dynkin}
\dynkinline{1}{0}{2}{0}
\dynkinline{2}{0}{3}{0}
\dynkindots{3}{0}{4}{0}
\dynkindoubleline{4}{0}{5}{0}
\dynkincross{1}{0}
\dynkincross{2}{0}
\dynkindot{3}{0}
\dynkindot{4}{0}
\dynkindot{5}{0}
\end{dynkin} \ar[dl]_{\eta} \ar[dr]^{\tau}& \\
G/R = \begin{dynkin}
\dynkinline{1}{0}{2}{0}
\dynkinline{2}{0}{3}{0}
\dynkindots{3}{0}{4}{0}
\dynkindoubleline{4}{0}{5}{0}
\dynkincross{1}{0}
\dynkindot{2}{0}
\dynkindot{3}{0}
\dynkindot{4}{0}
\dynkindot{5}{0}
\end{dynkin} & & G/P = \begin{dynkin}
\dynkinline{1}{0}{2}{0}
\dynkinline{2}{0}{3}{0}
\dynkindots{3}{0}{4}{0}
\dynkindoubleline{4}{0}{5}{0}
\dynkindot{1}{0}
\dynkincross{2}{0}
\dynkindot{3}{0}
\dynkindot{4}{0}
\dynkindot{5}{0}
\end{dynkin} . }
\end{equation}
For details about the Penrose transform, the reader can consult \cite{baston2016penrose}. Choose an open subset $X \subseteq G/P$ (the big affine cell, or a ball or a polydisc inside it), and define $Y=\tau^{-1}(X)$ and $Z=\eta(Y)$, so that we have the restricted double fibration:
\begin{equation}
\label{equation:rest_double_fib_1_2}
\xymatrix@R=1em{ & Y \ar[dl]_{\eta} \ar[dr]^{\tau}& \\
Z & & X . }
\end{equation}
The left-hand side of each of the fibrations is usualy called the twistor space, and the upper one the correspondence space. Start with the strictly $\mfrr$-dominant weight
\begin{equation}
\tilde{\lambda} + \rho = (k \pbar n-1, n-2, \ldots, 1), \qquad 1 \leq k \leq n-1
\end{equation}
(already defined in (\ref{definition:lambda_k_conj}), but we ease notation here), and consider the homogeneous sheaf $\mcaO_\mfrr(\tilde{\lambda})$ on the twistor space. We use the same notation for sheaves on (\ref{equation:double_fib_1_2}) and their restrictions to (\ref{equation:rest_double_fib_1_2}).

The topological inverse image sheaf $\eta^{-1} \mcaO_\mfrr(\tilde{\lambda})$ is constant on the fibers of $\eta$, so it admits a locally exact $G$-invariant resolution by homogeneous sheaves on $G/Q$, the so called relative BGG resolution. This is just the regular BGG resolution on each fiber. Since $\eta^{-1}(eR) \cong \begin{dynkin}
\dynkinline{1}{0}{2}{0}
\dynkindots{2}{0}{3}{0}
\dynkindoubleline{3}{0}{4}{0}
\dynkincross{1}{0}
\dynkinlabel{1}{0}{below}{$\alpha_2$}
\dynkindot{2}{0}
\dynkindot{3}{0}
\dynkindot{4}{0}
\dynkinlabel{4}{0}{below}{$\alpha_n$}
\end{dynkin}$, our relative BGG resolution is just (\ref{equation:BGG_iGr_1}) in rank $n-1$, while the first coordinate is kept fixed. More precisely, the sequence with the following ($\rho$-shifted) parameters resolves $\eta^{-1} \mcaO_\mfrr(\tilde{\lambda})$:
\begin{multline}
\label{equation:rel_BGG}
(k \pbar n-1 \pbar n-2,\ldots,1) \rightarrow (k \pbar n-2 \pbar n-1,n-3,\ldots,1) \rightarrow \ldots \\
\ldots \rightarrow (k \pbar 1 \pbar n-1,\ldots,2) \rightarrow (k \pbar -1 \pbar n-1,\ldots,2) \rightarrow (k \pbar -2 \pbar n-1,\ldots,3,1) \rightarrow \ldots \\ \ldots \rightarrow (k \pbar -n+1 \pbar n-2,\ldots,1) \to 0.
\end{multline}
Let us denote the $p$-th term in the above sequence by $\Delta_\eta^p$, starting from $p=0$.

The next step in the Penrose transform is to calculate higher direct images via $\tau$, denoted by $\tau_\ast^i$, of the sequence (\ref{equation:rel_BGG}). This is done using the (relative version of the) Bott-Borel-Weil theorem, which in our situation reads (in the $\rho$-shifted coordinates):
\begin{equation}
\label{equation_BBW}
\tau_\ast^i \, (\mu_1 \pbar \mu_2 \pbar \ldots) = \begin{cases}
(\mu_1,\mu_2) & \text{if } \mu_1>\mu_2, \ i=0,  \\
(\mu_2,\mu_1) & \text{if } \mu_2>\mu_1, \ i=1,\\
0 & \text{otherwise}.
\end{cases}.
\end{equation}
Consider the generic case $1<k<n-1$. Applying (\ref{equation_BBW}) to (\ref{equation:rel_BGG}) yields the so called hypercohomology spectral sequence $E_r^{pq}$, with $E_1^{pq}=\Gamma(X,\tau_\ast^q \Delta_\eta^p)$:
\begin{equation}
\label{equation:spect_seq_E1}
\scalebox{.86}{\spectralsequence{(n-1,k) \ar[r]^{d_1} & (n-2,k) \ar[r]^-{d_2} & \ldots \ar[rr]^-{d_{n-k-2}} & & (k+1,k) & 0  & 0 & & \ldots & & 0 \\
0 & 0 & \ldots & & 0 & 0 & (k,k-1) \ar[rr]^-{d_{n-k+1}} & & \ldots \ar[rr]^-{d_{2n-2}} & & (k,-n+1).}}
\end{equation}
Here both indices $p$ and $q$ start from $0$, and all the arrows are standard operators. Note that this is almost (\ref{equation:sing_BGG_k+}). The construction of the missing non-standard operator (\ref{equation:non_std_k+}) is done similarly as in the proof of \cite[Theorem 17]{mrden2017singular}; we describe it here.

Consider the part of the Čech bi-complex that calculates the higher direct images of (\ref{equation:rel_BGG}):
\[ \xymatrix@=1em{ & \vdots & \vdots & \vdots & \\
\ldots \ar[r] & \mathbf{\check{C}^{1}_{k+1}} \ar[r] \ar[u] & \check{C}^{1}_{k} \ar[r] \ar[u] & \check{C}^{1}_{k-1} \ar[u] \ar[r] & \ldots \\
\ldots \ar[r] & \check{C}^{0}_{k+1} \ar[r] \ar[u] & \check{C}^{0}_{k}  \ar[r] \ar[u] & \mathbf{\check{C}^{0}_{k-1}} \ar[u] \ar[r] & \ldots \\
\ldots \ar[r] & (k \pbar k+1 \pbar \ldots) \ar[r] \ar[u] & (k \pbar k \pbar \ldots) \ar[r] \ar[u] & (k \pbar k-1 \pbar \ldots) \ar[u] \ar[r] & \ldots } \]
Here the horizontal morphisms $d_h$ are induced from the differentials of the relative BGG resolution. The vertical morphisms $d_v$ are the usual differentials in the Čech resolution. We have $d_v^2=0$, $d_h^2=0$, and for each square, $d_h d_v =-d_v d_h$. By definition, the higher direct images are equal (locally) to the vertical cohomologies, ignoring the bottom row. The cochain spaces with nontrivial vertical cohomology are denoted in the bold font. All other vertical cohomologies are trivial, including the complete middle column. 

We will define the operator (\ref{equation:non_std_k+}) on the representatives of the vertical cohomology classes. Take a cocycle $x \in \mathbf{\check{C}^{1}_{k+1}}$. From $d_v d_h(x) = - d_h d_v(x) =0$ we see that $d_h(x) \in \check{C}^{1}_{k}$ is a cocycle. Since the vertical cohomology is trivial, it follows that $d_h(x) \in \Im d_v$. So, there is $y \in \check{C}^{0}_{k}$ such that $d_v(y) = d_h(x)$. Then, $d_h(y) \in \mathbf{\check{C}^{0}_{k-1}}$ is a cocycle: $d_v d_h (y) = - d_h d_v (y ) = - d_h^2 (x) =0$.
\[ \xymatrix@R=1em{x \in \mathbf{\check{C}^{1}_{k+1}} \ar@{|->}[r] \ar@{|-->}[rrd] & d_h (x) \in \check{C}^{1}_{k} \\
& y \in \check{C}^{0}_{k} \ar@{|->}[u] \ar@{|->}[r] & d_h (y) \in \mathbf{\check{C}^{0}_{k-1}} } \]
It is easy to see that the map $[x] \mapsto [d_h(y)]$ is well defined on the vertical cohomology classes, local and $G$-invariant, hence defines an invariant differential operator (\ref{equation:non_std_k+}). This makes (\ref{equation:sing_BGG_k+}) into a cochain complex, since $d_h^2=0$.

Deriving (\ref{equation:spect_seq_E1}) (i.e., taking the horizontal cohomologies) yields $E_2^{pq}$:
\begin{equation}
\label{equation:spect_seq_E2}
\spectralsequence{\Ker d_1 & \bullet  & \ldots  & \bullet & \Coker d_{n-k-2} \ar[rrd] & 0 & 0 & 0 & \ldots & 0 \\
0 & 0 & \ldots & 0 & 0 & 0 & \Ker d_{n-k+1} & \bullet & \ldots & \bullet ,} 
\end{equation}
where the only non-trivial morphism is induced from the just defined operator (\ref{equation:non_std_k+}). Deriving once more yields $E_3^{pq} = E_\infty^{pq}$ (there are no more non-trivial morphisms). The main property of this spectral sequence is that it converges to the cohomology on the twistor space (see \cite[9.1]{baston2016penrose}):
\begin{equation}
E_r^{pq} \ \Longrightarrow \ H^{p+q}(Z,\mcaO_\mfrr(\tilde{\lambda})).
\end{equation}

From this, and Lemma \ref{lemma:vanishing} bellow, we see that if $X$ is the big cell, all the bullets in (\ref{equation:spect_seq_E2}) are $0$ and the arrow is an isomorphism. Hence the sequence (\ref{equation:sing_BGG_k+}) is exact in positive degrees, and resolves $\Ker d_1 \cong H^1(Z,\mcaO_\mfrr(\tilde{\lambda}))$.

The order of an invariant differential operator is bounded by the difference of the generalized conformal weights (that is, the defining weight applied to the grading element) in the domain and the codomain, which is in our case $2$. See \cite{franek2008generalized}.

Completely the same arguments, but starting from the weight
\[ \tilde{\lambda} + \rho = (-k \pbar n-1, n-2, \ldots, 1), \qquad 1 \leq k \leq n-1, \]
give (\ref{equation:non_std_k-}) and (\ref{equation:sing_BGG_k-}). The proof for the boundary cases $k=1$ and $k=n-1$ is similar.
\end{proof}
\begin{lemma}
\label{lemma:vanishing}
Let $X \subseteq \iGr(2,2n)$ be the big affine cell, and $Z$ the corresponding twistor space as in (\ref{equation:rest_double_fib_1_2}). For any coherent sheaf $\mcaF$ on $Z$ we have
\[ H^i(Z,\mcaF) = 0, \qquad i \geq 2. \]
\end{lemma}
\begin{proof}
It is enough to cover $Z$ with two affine open subsets; the claim will then follow from Cartan's theorem B and Mayer-Vietoris exact sequence.

The canonical coordinates on the big affine cell are given by the composition $\mfru^- \stackrel{\exp}{\longrightarrow} U^- \longrightarrow G/P \cong \iGr(2,2n)$. Recall (\ref{equation:Delta_u}), and note that the negative root vectors give a coordinate system on $\mfru^-$; we write them as the corresponding roots. Then the composition above gives the following coordinates on $X$:
\begin{equation}
\label{equation:coord_big_cell}
\begin{pmatrix}
1 & 0 \\
0 & 1 \\
a_{13} & a_{23} \\
\vdots & \vdots \\
a_{1n} & a_{2n} \\
b_1 & c_{12}- \frac{1}{2} \sum_{j=3}^n a_{1j}c_{2j} + \frac{1}{2} \sum_{j=3}^n a_{2j}c_{1j} \\
c_{12}+ \frac{1}{2} \sum_{j=3}^n a_{1j}c_{2j} - \frac{1}{2} \sum_{j=3}^n a_{2j}c_{1j} & b_2 \\
c_{13} & c_{23} \\
\vdots & \vdots \\
c_{1n} & c_{2n} \\
\end{pmatrix}.
\end{equation}
Denote the columns above by $C_1$ and $C_2$. The matrix (\ref{equation:coord_big_cell}) represents the subspace $\lspan_\mbbC\{C_1,C_2\} \in \iGr(2,2n)$ (in a fixed symplectic basis). The space $G/Q$ can be modeled as the space of isotropic flags of type $(1,2)$ (as in (\ref{equation:flags})), and then the maps $\eta$ and $\tau$ become projections to the first and the second component respectively. So the twistor space $Z=\eta(\tau^{-1}(X)) \subseteq \iGr(1,2n)$ consists of all the lines passing through at least one plane of type (\ref{equation:coord_big_cell}):
\begin{equation}
\label{equation:twistor_space}
Z= \big\{ \lspan_\mbbC \{ \alpha \, C_1 + \beta \, C_2 \}  \ \colon \ \alpha,\beta \in \mbbC \text{ not both } 0 \big\}.
\end{equation}

We claim that this is equal to
\begin{equation}
Z' := \big\{ \lspan_\mbbC \{ \gamma \} \ \colon \ \gamma=\begin{pmatrix} \gamma_1 \\ \vdots \\ \gamma_{2n} \end{pmatrix}, \ \gamma_1=1 \big\} \cup \big\{ \lspan_\mbbC \{ \delta \} \ \colon \ \delta=\begin{pmatrix} \delta_1 \\ \vdots \\ \delta_{2n} \end{pmatrix}, \ \delta_2=1 \big\}
\end{equation}
Both of these subsets of $Z'$ are open, and affine (isomorphic to $\mbbC^{2n-1}$). Because in (\ref{equation:twistor_space}) we can take $\alpha=1$ or $\beta=1$, the inclusion $Z \subseteq Z'$ is clear. To see the converse, take $\gamma$ with $\gamma_1=1$ (it is similar for $\delta$). Then $1 \, C_1 + \gamma_2 \, C_2 = \gamma$ for the following values of the coordinates:
\begin{align*}
& a_{1j} = \gamma_j, \ a_{2j} = 0, c_{1j} = \gamma_{n+j}, \ c_{2j} = 0, \quad j=3,\ldots,n, \\
& b_1 = \gamma_{n+1} - \gamma_2 \cdot \gamma_{n+2}, \ c_{12}=\gamma_{n+2}, \ b_2=0.
\end{align*}
\end{proof}

\subsection{Description of $d_1$}

Each of the singular BGG complexes resolves the kernel of a certain invariant differential operator (the first operator in  the complex) that we called $d_1$. This kernel turned out to be isomorphic to certain cohomology on the twistor space. Here we will describe these operators for the minimal cases, in terms of maximal vectors of the corresponding generalized Verma modules.

Recall that an invariant differential operator $d \colon \mcaO_\mfrp(\lambda) \to \mcaO_\mfrp(\mu)$ corresponds to a homomorphism $f \colon M_\mfrp(\mu) \to M_\mfrp(\lambda)$, which in turn, is completely determined by the image of the highest weight vector $1 \otimes v_\mu$. This image must be a maximal vector $v_{\text{max}} \in M_\mfrp(\lambda)$ of weight $\mu$, that is, annihilated by all the positive root vectors. This is computable in particular cases, but not in general. To go back from $f$ to $d$, take a section of $\mcaO_\mfrp(\lambda)$. Such a section corresponds to a $P$-equivariant  $E_\mfrp(\lambda)$-valued function $s$ defined locally on $G$ (suppose around $e$, without loss of generality). Take $v \in F_\mfrp(\mu)$, and write $f(1\otimes v)$ as a finite sum of elements $u_1 \ldots u_k  \otimes w$, where $u_i \in \mfru^-$ and $w \in F_\mfrp(\lambda)$. Then $(ds)(e) \in E_\mfrp(\mu) = F_\mfrp(\mu)^\ast$ evaluated at $v$ is the finite sum of terms $((L_{u_1} \ldots L_{u_k}  s)(e))(w)$, where $L_{u_i}$ are the left-invariant vector fields. For details, see \cite[A]{cap2001bernstein}.

For each $d_1$ we will write down the corresponding maximal vector, which determines $d_1$ as described above. Let us fix some notation:
\begin{itemize}
\item For $\gamma \in \Delta^+(\mfrg,\mfrh)$ denote by $Y_\gamma$ the standard root vector for $-\gamma$.

\item Denote by $W_{(\lambda_1,\lambda_2)}$ the representation of $\mfrgl(2,\mbbC) \subseteq \mfrl$ (the first node in the Dynkin diagram) with highest weight $(\lambda_1,\lambda_2)$, by $w_{\lambda_1-\lambda_2}$ its highest weight vector, and define $w_{i-2} := 2/(i+\lambda_1-\lambda_2) \cdot Y_{a_{12}} \, w_{i}$ inductively.

\item Denote by $V$ the standard representation of $\mfrsp(2n-4,\mbbC) \subseteq \mfrl$ (the last $n-2$ nodes) with a symplectic basis $\{e_3,\ldots,e_n,f_3,\ldots,f_n\}$.

\item Declare that $\mfrgl(2,\mbbC)$ acts trivially on $V$, and also $\mfrsp(2n-4,\mbbC)$ on $W_{(\lambda_1,\lambda_2)}$.
\end{itemize} 

For $0<k<n-2$ in the positive BGG complex, we have $(n-1,k) \stackrel{d_1}{\longrightarrow} (n-2,k)$. By subtracting $\rho$ and switching to generalized Verma modules, $d_1$ corresponds to
\[ M_\mfrp(-2,k-n+1 \pbar 1,0, \ldots, 0) \rightarrow  M_\mfrp(-1,k-n+1 \pbar 0,0, \ldots, 0) = U(\mfru^-) \otimes W_{(-1,k-n+1)}. \]
The vectors of the correct weight in the codomain are spanned by $Y_{a_{13}} \otimes w_{n-k-2}$ and $Y_{a_{23}} \otimes w_{n-k-4}$. It is not hard to find that $Y_{a_{13}} \otimes w_{n-k-2} - Y_{a_{23}} \otimes w_{n-k-4}$ is the linear combination that is annihilated by the positive root vectors (it is enough to consider only the simple roots). So, this is the maximal vector that determines $d_1$ in this case. We list all the cases in Table \ref{table:max_vect}.

\begin{table}[ht]
\begin{tabular}{|l||l|l|}
\hline 
 & $F_\mfrp(\lambda)$ & $v_\text{max}$ for $d_1$ in the positive BGG complexes \\ 
\hline\hline 
$0<k<n-2$ & $W_{(-1,k-n+1)}$ & $Y_{a_{13}} \otimes w_{n-k-2} - Y_{a_{23}} \otimes w_{n-k-4}$ \\ 
\hline 
$n>3$, $k=n-1$ &  $W_{(-1,-1)} \otimes V$ & $ Y_{a_{24}} \otimes 1 \otimes e_3 - Y_{a_{23}} \otimes 1 \otimes e_4 $ \\ 
\hline 
$n>3$, $k=n-2$ &  $W_{(-1,-1)}$ & $(Y_{a_{13}}Y_{a_{24}} - Y_{a_{14}}Y_{a_{23}})\otimes 1$ \\ 
\hline 
$n=3$, $k=2$ &  $W_{(-1,-1)} \otimes V$ & $(Y_{c_{23}}Y_{a_{23}} - 4Y_{b_2}) \otimes 1 \otimes e_3 - Y_{a_{23}}^2 \otimes 1 \otimes e_4 $ \\ 
\hline
$n=3$, $k=1$ &  $W_{(-1,-1)} $ & $(Y_{a_{23}}^2 Y_{c_{13}} - Y_{c_{23}} Y_{a_{23}} Y_{a_{13}}  - 4 Y_{a_{13}}  Y_{b_{2}} ) \otimes 1 $ \\ 
\hline\hline 
 & $F_\mfrp(\lambda)$ & $v_\text{max}$ for $d_1$ in the negative BGG complexes\\ 
\hline\hline
$0<k<n$ & $W_{(-1,-n-k-1)}$ & $Y_{a_{13}}\otimes w_{n+k-2} + Y_{a_{23}}\otimes w_{n+k-4}$ \\
\hline
\end{tabular}
\medskip
\caption{The first operators in the BGG complexes}
\label{table:max_vect}
\end{table}

\subsection{Comments on the remaining cases}

We have obtained only partial results in the ``long root singular'' ($k=0$) case. There seem to be two non-standard operators, which connect the connected components in Figure \ref{figure:singular_orbit_0}. Conjecturally, the singular BGG complex in this case should have the following shape (similar to the even-orthogonal case, see \cite{krump2006singular}):
\[ \xymatrix@C=1em@R=.5em{ (n-1,0) \ar[r] & \ldots \ar[r] & (2,0)\ar[rd] \ar[r] & (1,0)\ar@{}[d]|{\oplus} \ar[rd] \\
& & & (0,-1) \ar[r] & (0,-2) \ar[r] & \ldots \ar[r] & (0,-n+1)}. \]
New methods are needed to obtain these BGG complexes, which authors are currently working on. Also, weights with multiple singularities (for example $(2211)$ and $(2210)$) have much smaller singular orbits, but somehow evade this kind of Penrose transform. All such (minimal integral) weights define scalar-generalized Verma modules, so homomorphisms between them are classified in \cite{matumoto2006homomorphisms}.

Finally, we note that the methods presented here are likely to be applicable to the isotropic Grassmannians $\iGr(s,2n)$ for $s>2$, at least in highly-singular cases.

\bibliographystyle{alpha}
\bibliography{My_BibTex_Library}

\end{document}